\pdfoutput=1
\newif\ifdraft \draftfalse

\drafttrue 

\documentclass{article}

\usepackage{parskip} 
\usepackage{url} 
\usepackage{natbib} 
\bibpunct{(}{)}{;}{a}{,}{,} 

\usepackage{amssymb}
\usepackage{amsfonts}
\usepackage{amsmath}
\usepackage{amsthm}
\usepackage{graphicx}
\usepackage{longtable}
\usepackage{dcolumn}
\usepackage{tabularx}
\usepackage{rotating}
\usepackage{xtab}
\usepackage{paralist} 
\usepackage{verbatim}
\usepackage{subfiles} 
\usepackage[utf8]{inputenc}
\usepackage{xspace}
\usepackage{bbm}
\usepackage{algorithm}
\usepackage[noend]{algorithmic}\usepackage{color}
\usepackage[]{hyperref}
\usepackage{bookmark}
\usepackage[capitalize]{cleveref}
\usepackage{ltxcmds}
\usepackage{parskip} 
\usepackage{subcaption}
\usepackage[affil-it]{authblk}
\usepackage{fullpage}
\usepackage[T1]{fontenc}    
\usepackage{booktabs}       
\usepackage{nicefrac}       
\usepackage{microtype}      

\newtheorem{theorem}{Theorem}[section]
\makeatletter
\def\thmhead@plain#1#2#3{%
  \thmname{#1}\thmnumber{\@ifnotempty{#1}{ }\@upn{#2}}%
  \thmnote{ {\the\thm@notefont [#3]}}}
\let\thmhead\thmhead@plain
\makeatother
\newtheorem{definition}[theorem]{Definition}

\newtheorem{lemma}[theorem]{Lemma}

\newtheorem{assumpt}[theorem]{Assumption}

\newcommand\R{\mathbb{R}}

\newcommand{\cD}{\mathcal{D}}

\newcommand{\cM}{\mathcal{M}}

\newcommand{\cX}{\mathcal{X}}

\newcommand{\cY}{\mathcal{Y}}

\newcommand{\bU}{\pmb{U}}

\newcommand{\bV}{\pmb{V}}

\newcommand{\bX}{\pmb{X}}

\newcommand{\bZ}{\pmb{Z}}

\newcommand{\bH}{\pmb{H}}

\newcommand{\nullhyp}{\mathrm{H}_0}
\newcommand{\althyp}{\mathrm{H}_1}

\newcommand{\bbe}{\pmb{e}}

\newcommand{\bbx}{\pmb{x}}

\newcommand{\bbz}{\pmb{z}}
\newcommand{\bp}{\pmb{p}}

\newcommand{\nullbp}{\bp^0}
\newcommand{\nullp}{p^0}
\newcommand{\mult}{\mathrm{Multinomial}}
\newcommand{\diag}{\mathrm{Diag}}
\newcommand{\Normal}[2]{\mathrm{N}\left(#1,#2 \right)}

\newcommand{\chisqGLRV}[1]{\tilde{\mathrm{T}}\left(#1 \right)}
\newcommand{\chisqKR}[1]{{\mathrm{T}}_{KR}^{(n)}\left(#1 \right)}
\newcommand{\bchisqKR}[1]{\pmb{\hat{\mathrm{T}}}^{(n)}_{KR}\left(#1 \right)}
\newcommand{\transpose}{\intercal}

\newcommand{\multRR}{\alg_\mathtt{EXP}}

\newcommand{\btheta}{\pmb{\theta}}
\newcommand{\bDelta}{\pmb{\Delta}}

\newcommand{\hbpipi}[1]{\hat{\pmb\pi}^{(#1)}}

\newcommand{\tbpipi}[1]{\tilde{\pmb\pi}^{(#1)}}

\newcommand{\bpipi}[1]{\pmb{\pi}^{(#1)}}
\newcommand{\cbpipi}[1]{\check{\pmb{\pi}}^{(#1)}}

\newcommand{\stat}{\mathrm{T}}

\newcommand{\covarJL}{\Sigma}

\newcommand{\one}{\pmb{1}}
\newcommand{\ExpGOF}{\texttt{LocalExpGOF}}
\newcommand{\ExpIND}{\texttt{LocalExpIND}}
\newcommand{\LapGOF}{\texttt{LocalLapGOF}}
\newcommand{\GaussGOF}{\texttt{LocalNoiseGOF}}
\newcommand{\GaussIND}{\texttt{LocalNoiseIND}}
\newcommand{\chisqExp}[1]{\stat_{\mathrm{Exp}}^{(n)}\left(#1\right)} 
\newcommand{\bchisqExp}[1]{\check{\pmb{\stat}}_{\mathrm{Exp}}^{(n)}\left(#1\right)} 
\newcommand{\BitGOF}{\texttt{LocalBitFlipGOF}}
\newcommand{\BitIND}{\texttt{LocalBitFlipIND}}
\newcommand{\Mbit}{\cM_{\text{bit}}}
\newcommand{\chisqBF}[1]{\stat_{\text{BitFlip}}^{(n)}\left(#1\right)} 
\newcommand{\bchisqBF}[1]{\tilde{\pmb{\stat}}_{\text{BitFlip}}^{(n)}\left(#1\right)} 

\newcommand{\alg}{\cM}
\renewcommand{\tilde}{\widetilde}
\renewcommand{\hat}{\widehat}

\newcommand\defn{\mathrm{defn}}

\newcommand{\Lap}{\mathrm{Lap}}

\DeclareMathOperator*{\myargmin}{\arg\!\min}

\newcommand{\1}{\mathbbm{1}}
\newcommand{\var}[1]{\mathbb{V}\left[#1\right]}

\newcommand{\ex}[1]{\mathbb{E}\left[#1\right]}
\DeclareMathOperator*{\Expectation}{\mathbb{E}}
\newcommand{\Ex}[2]{\Expectation_{#1}\left[#2\right]}

\newcommand{\prob}[1]{\mathrm{Pr}\left[#1\right]}

	\title{Local Private Hypothesis Testing: Chi-Square Tests}
         \author{}

	\author[1]{Marco Gaboardi}
	\author[2]{Ryan Rogers}
	\affil[1]{University at Buffalo, The State University of New York}
	\affil[2]{University of Pennsylvania}
\begin{document}

	\maketitle

\begin{abstract}
The local model for differential privacy is emerging as the reference
model for practical applications of collecting  and sharing sensitive information while satisfying
strong privacy guarantees.  In the local model, there is no \emph{trusted}
 entity which is allowed to have each individual's raw data as is
 assumed in the traditional \emph{curator} model for differential
 privacy. Individuals' data are usually perturbed before sharing them. 
We explore the design of private hypothesis tests in the local model,
where each data entry is perturbed to ensure the privacy of each
participant.  Specifically, we analyze locally private chi-square
tests for goodness of fit and independence testing, which have been
studied in the traditional, curator model for differential privacy.
 \end{abstract}
 
 \section{Introduction}
Hypothesis testing is a widely applied statistical tool used to test
whether given models should be rejected, or not, based on sampled data from a population.  Hypothesis testing was initially developed for scientific and survey data, but today it is also an essential tool to test models over collections of social network, mobile, and crowdsourced data~\citep{asa14,HGH08,Steele20160690}.
Collected data samples may contain highly sensitive information about
the subjects, and the privacy of individuals can be compromised when
the results of a data analysis are released. 
 In work from \cite{Homer08}, it was shown that a subject in a dataset
 can be identified as being in the case or control group based on the
 aggregate statistics of a genetic-wide association study
 (GWAS). Privacy-risks may bring data contributors to opt out, which
 reduces the confidence in the data study.
A way to address this concern is by developing new techniques to
support privacy-preserving data analysis. Among the different
approaches, differential privacy~\citep{DMNS06} has emerged as a viable
solution: it provides strong privacy guarantees and it allows to
release accurate statistics. A standard way to achieve differential
privacy is by injecting some statistical noise in the computation of
the data analysis. When the noise is carefully chosen, it helps
to protect the individual privacy without compromising the utility of
the data analysis.
Several recent works have studied differentially private hypothesis
tests that can be used in place of the standard,
non-private hypothesis tests~\citep{USF13,YFSU14,Shef15,KS16,WLK15, GLRV16,KR17,CaiDK17}. These tests work in the \emph{curator
model} of differential privacy. In this model, the data is
centrally stored and the curator carefully injects noise in the computation
of the data analysis in order to satisfy differential
privacy.

In this work we instead address the \emph{local model} of privacy,
formally introduced by \citet{KLNRS08}.  The first differentially private algorithm called \emph{randomized response} -- in fact it predates the definition of differential privacy by more than 40 years -- guarantees differential privacy in the local model \citep{Warner65}.  In this model, there is no \emph{trusted} centralized entity which is
responsible for the noise injection. Instead, each individual adds
enough noise to guarantee differential
privacy for their own data, which provides a stronger privacy guarantee when compared to traditional differential privacy. The data analysis is then run over the collection of the
individually sanitized data.  
The local model of differential privacy is a convenient model for
several applications: for example it is used to 
collect statistics about the activity of the Google Chrome Web
browser users~\citep{EPK14}, and to collect statistics about the typing
patterns of Apple's iPhone users~\citep{appleDP2016}.  
Despite these applications, the local model has received far less
attention than the more standard centralized curator model. This
is in part due to the more firm requirements imposed by this
model, which make the design of effective data analysis harder. 

Our main contribution is in designing chi-square hypothesis tests for
the local model of differential privacy. 
Similar to
previous works we focus on goodness of fit and independence hypothesis
tests.  Most of the differentially private chi-square hypothesis tests
proposed so far are based on mechanisms that add noise in
some form to
the aggregate data, e.g. the cells of the contingency tables, or the
resulting chi-square statistics value. These approaches cannot be used
in the local model, since noise needs to be added at the individual's
data level. We then consider instead general privatizing
techniques in the local model, and we study how to build new hypothesis tests with them. 
Each test we present is characterized by a specific local model mechanism.
The main technical challenge for designing each test is to create statistics, which incorporate the local model mechanisms, that converge as we collect more data to a chi-square distribution, as in the classical chi-square tests.  We then use these statistics to 
find the critical value to correctly bound the Type I error.

We present three different goodness of fit tests:
$\GaussGOF$ presents a statistic that guarantees the convergence to a
chi-square distribution under the null hypothesis so that we can use the correct critical values when local (concentrated) differential privacy is guaranteed by adding Laplace or Gaussian noise to the individual
data;
$\ExpGOF$ also provides a statistic that converges to a chi-square under the null hypothesis when a private value for each individual is selected by using the exponential mechanism \citep{MT07};
finally, $\BitGOF$ introduces a statistic that converges to a chi-square distribution when the data is privatized using a bit flipping algorithm \citep{BS15}, which provide better accuracy for higher dimensional data.  Further, we develop corresponding independence tests: $\GaussIND$, $\ExpIND$, and $\BitIND$.  
For all these tests we study their asymptotic behavior.  A desiderata
for private hypothesis tests is to have a guaranteed upper bound on the
probability of a false discovery (or Type I error) -- rejecting a null
hypothesis or model when the data was actually generated from it -- and
 to minimize the probability of a Type II error, which is
failing to reject the null hypothesis when the model is indeed false. This
latter criteria corresponds to the \emph{power} of the statistical test. We then
present experimental results showing the power of the
different tests which demonstrates that no single local differentially private algorithm is best across all data dimensions and privacy parameter regimes.
\section{Related Works}
There have been several works in developing private hypothesis test for categorical data, but all look at the traditional model of (concentrated) differential privacy instead of the local model, which we consider here. 
Several works have explored private statistical inference for GWAS data, \citep{USF13,YFSU14,JS13}.  Following these works, there has also been general work in private chi-square hypothesis tests, where the main tests are for goodness of fit and independence testing, although some do extend to more general tests \citep{WLK15,GLRV16,KR17,CaiDK17,KFS17}.  
Among these, the works most related to ours are the ones by \citep{GLRV16,KR17}. We will compare our work with these in Section~\ref{sec:hyp_test} after introducing them.
There has also been work in private hypothesis testing for ordinary least squares regression \citep{Shef15}.

There are other works that have studied statistical inference and estimators in the local model of differential privacy.  
\citet{DJW13NIPS, DJW13FOCS} focus on controlling disclosure risk in statistical estimation and inference by ensuring the analysis satisfies local differential privacy.  They provide minimax convergence rates to show the tight tradeoffs between privacy and statistical efficiency, i.e. the number of samples required to give quality estimators.  
In their work, they show that a generalized version of randomized response gives optimal sample complexity for estimating the multinomial probability vector. We use this idea as the basis for our hypothesis test $\BitGOF$.
\citet{KOV14} also considers hypothesis testing in the local model, although they measure utility in terms of $f$-divergences and do not give a decision rule, i.e. when to reject a given null hypothesis.  We provide statistics whose distributions asymptotically follow a chi-square distribution, which allows for approximating statistical $p$-values that can be used in a decision rule.  We consider their \emph{extremal} mechanisms and empirically confirm their result that for small privacy regimes (small $\epsilon$) one mechanism has higher \emph{utility} than other mechanisms and for large privacy regimes (large $\epsilon$) a different mechanism outperforms the others.  However, we measure utility in terms of the power of a locally private hypothesis test subject to a given Type I error bound.  Other notable works in the local privacy model include \citet{PG16, KBR16, YB17}

 \section{Preliminaries}
We consider datasets $\bbx = (x_1, \cdots, x_n) \in \cX^n$ in some data universe $\cX$, typically $\cX = \{0,1\}^d$ where $d$ is the dimensionality.  We first present the standard definition of differential privacy, as well as its variant \emph{concentrated differential privacy}.  We say that two datasets $\bbx, \bbx' \in \cX^n$ are \emph{neighboring} if they differ in at most one element, i.e. $\exists i \in [n]$ such that $x_i \neq x_i'$ and $\forall j \neq i$, $x_j = x_j'$.

\begin{definition}[\cite{DMNS06,DKMMN06}]
An algorithm $\cM: \cX^n \to \cY$ is $(\epsilon,\delta)$-differentially private (DP) if for all neighboring datasets $\bbx, \bbx' \in \cX^n$ and for all outcomes $S \subseteq \cY$, we have
$
\prob{\cM(\bbx) \in S} \leq e^\epsilon\prob{\cM(\bbx') \in S} + \delta.
$
\end{definition}
We then state the definition of zero-mean concentrated differential privacy.
\begin{definition}[\cite{BS16}]
An algorithm $\cM: \cX^n \to \cY$ is $\rho$-zero-mean concentrated differentially private (zCDP) if for all neighboring datasets $\bbx, \bbx' \in \cX^n$, we have the following bound for all $t>0$ where the expectation is over outcomes $y \sim \cM(\bbx)$,
$$
\Ex{}{\exp\left(t \left( \ln\left( \frac{\prob{\cM(\bbx) = y}}{\prob{\cM(\bbx') = y} }\right) - \rho\right)\right)} \leq e^{t^2\rho}.
$$
\end{definition}

Note that in both of these privacy definitions, it is assumed that all the data is stored in a central location and the algorithm $\cM$ can access all the data.  Most of the work in differential privacy has been in this \emph{trusted curator model}.  
One of the main reasons for this is that we can achieve much greater accuracy in our differentially private statistics when used in the curator setting.  However, in many cases, having a trusted curator is too strong of an assumption. 
 We then define \emph{local} differential privacy, formalized by \cite{KLNRS08} and \cite{DR14}, which does not require the subjects to release their raw data, rather each data entry is perturbed to prevent the true entry from being stored.  Thus, local differential privacy ensures a very strong privacy guarantee.

\begin{definition}[LR Oracle]
Given a dataset $\bbx$, a \emph{local randomizer oracle} $LR_{\bbx}(\cdot, \cdot)$ takes as input an index $i \in[n]$ and an $\epsilon$-DP algorithm $R$, and outputs $y \in \cY$ chosen according to the distribution of $R(x_i)$, i.e. $LR_{\bbx}(i,R) = R(x_i)$.  
\end{definition}

\begin{definition}[\citet{KLNRS08}]
An algorithm $\alg: \cX^n \to \cY$ is $(\epsilon,\delta)$-\emph{local differentially private} (LDP) if it accesses the input database $\bbx$ via the LR oracle $LR_{\bbx}$ with the following restriction: if $LR(i,R_j)$ for $j \in [k]$ are $\alg$'s invocations of $LR_{\bbx}$ on index $i$, then each $R_j$ for $j \in [k]$ is  $(\epsilon_j,\delta_j)$- DP and $\sum_{j=1}^k \epsilon_j \leq \epsilon$, $\sum_{j=1}^k \delta_j \leq \delta$.
\end{definition}
An easy consequence of this definition is that an algorithm which is $(\epsilon,\delta)$-LDP is also $(\epsilon,\delta)$-DP.  Note that these definitions can be extended to include $\rho$-local zCDP (LzCDP) where each local randomizer is $\rho_j$-zCDP and $\sum_{j = 1}^k \rho_j \leq \rho$.  We point out the following connection between $LzCDP$ and $LDP$, which follows directly from results in \cite{BS16}
\begin{lemma}
If $\cM: \cX^n \to \cY$ is $(\epsilon,0)$-LDP then it is also $\epsilon^2/2$-LzCDP.  If $\cM$ is $\rho$-LzCDP, then it is also $\left(\left(\rho + \sqrt{2\rho\ln(2/\delta)} \right),\delta\right)$-LDP for any $\delta >0$.  
\end{lemma}

Thus, in the local setting, \emph(pure) LDP (where $\delta = 0$) provides the strongest level of privacy, followed by LzCDP and then \emph{approximate}-LDP (where $\delta >0)$.  
\section{Chi-Square Hypothesis Tests}
\label{sec:hyp_test}
As was studied in \cite{GLRV16}, \cite{WLK15}, and \cite{KR17}, we will study hypothesis tests with categorical data.  A null hypothesis, or model $\nullhyp$ is how we might expect the data to be generated.  The goal for hypothesis testing is to reject the null hypothesis if the data is not likely to have been generated from the given model.  As is common in statistical inference, we want to design hypothesis tests to bound the probability of a false discovery (or Type I error), i.e. rejecting a null hypothesis when the data was actually generated from it, by at most some amount $\alpha$, such as $5\%$.  However, designing tests that achieve this is easy, because we can just ignore the data and always \emph{fail to reject} the null hypothesis, i.e. have an inconclusive test.  Thus, we would additionally like to design our tests so that they can reject $\nullhyp$ if the data was not actually generated from the given model.  We then want to minimize the probability of a Type II error, which is failing to reject $\nullhyp$ when the model is false, subject to a given Type I error.

For goodness of fit testing, we assume that each individual's data $\bX_i$ for $i \in [n]$ is sampled i.i.d. from $\mult(1,\bp)$ where $\bp\in \R^d_{> 0}$ and $\bp^\transpose \cdot \one =1$.   The classical chi-square hypothesis test (without privacy) forms the histogram $\bH = (H_1, \cdots, H_d) = \sum_{i=1}^n \bX_i$ and computes the chi-square statistic
$
\stat = \sum_{j =1}^d \frac{\left(H_j - n \nullp_j\right)^2}{n \nullp_j}.
$
 The reason for using this statistic is that it converges in distribution to $\chi_{d-1}^2$ as more data is collected, i.e. $n \to \infty$, when $\nullhyp: \bp =   \nullbp$ holds.  Hence, we can ensure the probability of false discovery to be close to $\alpha$ as long as we only reject $\nullhyp$ when $\stat > \chi^2_{d-1,1-\alpha}$ where the \emph{critical value} $\chi^2_{d-1,1-\alpha}$ is defined as the following quantity $\prob{\chi^2_{d-1} > \chi^2_{d-1,1-\alpha}} = \alpha$.  

\subsection{Prior Private Chi-square Tests in the Curator Model}

One approach for chi-square private hypothesis tests that was explored by \cite{GLRV16} and \cite{WLK15} is to add noise (Gaussian or Laplace) directly to the histogram to ensure privacy and then use the classical test statistic.  Note that the resulting asymptotic distribution needs to be modified for such changes to the statistic -- it is no longer a chi-square random variable.  To introduce the different statistics, we will consider goodness of fit testing after adding noise $\bZ$ from distribution $\cD^n$ to the histogram of counts $\tilde{\bH} = \bH + \bZ$, which ensures $\rho$-zCDP when $\cD = \Normal{0}{1/\rho}$ and $\epsilon$-DP when $\cD = \Lap(2/\epsilon)$.  The chi-square statistic then becomes
\begin{equation}
\chisqGLRV{\cD} = \sum_{i=1}^d \frac{\left(H_i + Z_i  - n \nullp_i\right)^2}{n\nullp_i} \quad \text{ where } \bZ \sim \cD^n.
\label{eq:GLRV}
\end{equation}
They then show that this statistic converges in distribution to a linear combination of chi-squared random variables, when $\cD \sim \Normal{0}{1/\rho}$ and $\rho$ is also decreasing with $n$.

In followup work from \cite{KR17}, the authors showed that modifying the chi-square statistic to account for the additional noise leads to tests with better empirical power.  The \emph{projected} statistic from \cite{KR17} is the following where we use projection matrix $\Pi \stackrel{\defn}{=} \left(I_d - \frac{1}{d} \one\one^\transpose \right)$, middle matrix $M_\sigma =  \Pi \left( \diag\left(\nullbp+\sigma \right) - \nullbp\left(\nullbp\right)^\transpose \right)^{-1} \Pi$, and sample noise $\bZ \sim \cD^n$,
\begin{equation}
 \chisqKR{\sigma;\cD} = n \left( \frac{\bH+ \bZ}{n}  -  \nullbp\right)^\transpose M_\sigma \left( \frac{\bH +\bZ}{n}  -  \nullbp\right).
\label{eq:KR}
\end{equation}

We use $\cD = \Lap(2/\epsilon)$ with $\sigma = \frac{8}{n \epsilon^2}$ for an $\epsilon$-DP claim or $\cD = \Normal{0}{1/\rho}$ with $\sigma = \frac{1}{n \rho}$ for a $\rho$-zCDP claim. When comparing the power of all our tests, we will be considering the alternate $\althyp: \bp = \bp^1_n$ where 
\begin{equation}
\bp^1_n = \nullbp + \frac{\bDelta}{\sqrt{n}} \qquad \text{ where } \one^\transpose \bDelta = 0.
\label{eq:althyp}
\end{equation}

\begin{theorem}[\cite{KR17}]
Under the null hypothesis $\nullhyp: \bp = \nullbp$, the statistic $\chisqKR{\frac{1}{n\rho}; \Normal{0}{1/\rho}}$ given in \eqref{eq:KR} for $\rho>0$ converges in distribution to $\chi^2_{d-1}$.  Further, under the alternate hypothesis $\althyp:\bp = \bp^{1}_n$, 
the resulting asymptotic distribution is a noncentral chi-square random variable,
\begin{align*}
& \chisqKR{\frac{1}{n\rho} ; \Normal{0}{1/\rho}} \stackrel{D}{\to}\\
& \qquad  \chi_{d-1}^2\left( \bDelta^\transpose \left(\diag(\nullbp) - \nullbp \left(\nullbp \right)^\transpose+ 1/\rho I_d \right)^{-1} \bDelta\right).
\end{align*}
\label{thm:KRtest}
\end{theorem}
When $\cD = \Lap(2/\epsilon)$, we can still obtain the null hypothesis distribution using Monte Carlo simulations to estimate the critical value, since the asymptotic distribution will no longer be chi-square.  That is, we can obtain $m$ samples from the statistic under the null hypothesis with Laplace noise added to the histogram of counts.  We can then guarantee that the probability of a false discovery is at most $\alpha$ as long as $m > \lceil 1/\alpha \rceil$ (see \cite{GLRV16} for more details). 
 \section{Local Private Chi-Square Goodness of Fit Tests}
We now turn to designing local private goodness of fit tests.  We begin by showing how the existing statistics from the previous section can be used in the local setting and then develop new tests based on the \emph{exponential mechanism} \citep{MT07} and bit flipping \citep{BS15}.  Each test is locally private because it perturbs each individual's data through a local randomizer.  
We finish the section by empirically checking the power of each test to see which tests outperform others in different parameter regimes.  We empirically show that the power of a test is directly related to the size of the noncentral parameter of the chi-square statistic under the alternate distribution.

\subsection{Goodness of Fit Test with Noise Addition}
In the local model we can add $\bZ_{i} \sim \Normal{\pmb{0}}{\frac{1}{\rho}\ I_d}$ independent noise to each individual's data $\bX_i$ to ensure $\rho$-LzCDP or $\bZ_{i} \stackrel{i.i.d.}{\sim} \Lap\left(\frac{2}{\epsilon}\right)$ independent noise to $\bX_i$ to ensure $\epsilon$-LDP.    In either case, the resulting noisy histogram $\hat{\bH} = \bH + \bZ$ where $\bZ = \sum_i \bZ_i$ will have variance that scales with $n$ for fixed privacy parameters $\epsilon,\rho>0$.   Consider the case where we add Gaussian noise, which results in the following histogram, $\hat{\bH} = \bH + \bZ$ where $\bZ \sim \Normal{\pmb{0}}{\frac{n}{\rho} \ I_d}$.  Thus, we can use either statistic $\chisqGLRV{\rho/n}$ or $\chisqKR{\rho/n}$, with the latter statistic typically having better empirical power \citep{KR17}.  We then give our first local private hypothesis test in \Cref{alg:LocalGaussGOF}.
\begin{algorithm}
\caption{Locally Private GOF Test:$\GaussGOF$}
\label{alg:LocalGaussGOF}
\begin{algorithmic}
\REQUIRE $\bbx = (\bbx_1,\cdots, \bbx_n)$, $\rho$, $\alpha$, $\nullhyp: \bp = \nullbp$.
\STATE Let $\bH = \sum_{\ell = 1}^n \bbx_\ell$ 
\IF{$\cD = \Normal{0}{n/\rho}$}
	\STATE Set $ q = \chisqKR{1/\rho;\cD }$ given in \eqref{eq:KR}.
	\STATE {\bf if }$q > \chi_{d-1,1-\alpha}^2$
	 Decision $\gets $ Reject.
	\STATE {\bf else }
	 Decision $\gets$ Fail to Reject.
\ENDIF
\IF{$\cD =\sum_{i=1}^n \Lap(2/\epsilon)$}
	\STATE Set $ q = \chisqKR{8/\epsilon^2;\cD }$ given in \eqref{eq:KR}.
	\STATE Sample $m > \lceil 1/\alpha \rceil$ from the distribution of  $\chisqKR{8/\epsilon^2;\cD }$ assuming $H_0$
	\STATE Set $\tau$ to be the $\lceil (m+1)(1- \alpha)\rceil$th largest sample.
	\STATE {\bf if }$q > \tau$
	 Decision $\gets $ Reject.
	\STATE {\bf else }
	 Decision $\gets$ Fail to Reject.
\ENDIF
\ENSURE Decision
\end{algorithmic}
\end{algorithm}

\begin{theorem}
$\GaussGOF$ is $\rho$-LzCDP when $\cD = \Normal{0}{1/\rho}$ and $\epsilon$-LDP when $\cD = \Lap(2/\epsilon)$.  
\end{theorem}
 \begin{proof}
The proof follows from the fact that we are adding appropriately scaled noise to each individual's data via the Gaussian mechanism and then $\GaussGOF$ aggregates the privatized data, which is just a post-processing function on the privatized data.  
\end{proof}

Although we cannot guarantee the probability of a Type I error at most $\alpha$ due to the fact that we use the asymptotic distribution (as in the tests from prior work and the classical chi-square tests without privacy), we expect the Type I errors to be similar to those from the nonprivate test.
Note that the test can be modified to accommodate arbitrary noise distributions, e.g. Laplace to ensure differential privacy as was done in \cite{KR17}.  In this case, we can use a Monte Carlo (MC) approach to estimate the critical value $\tau$ that ensures the probability of a Type I error is at most $\alpha$ if we reject $\nullhyp$ when the statistic is larger than $\tau$.  For the local setting, if each individual perturbs each coordinate by adding $\Lap\left(2/\epsilon\right)$ then this will ensure our test is $\epsilon$-LDP.  However, the sum of independent Laplace random variables is not Laplace, so we will need to estimate a sum of $n$ independent Laplace random variables using MC.  In the experiments section we will compare the other local private tests with the version of $\GaussGOF$ which uses Laplace noise and samples $m$ entries from the exact distribution under $\nullhyp$ to find the critical value.

Rather than having to add noise to each component of the original data histogram, we consider applying randomized response to obtain a LDP hypothesis test.  We will use a form of the \emph{exponential mechanism} \citep{MT07} given in \Cref{alg:multRR} which takes a single data entry from the set $\{\bbe_1,\cdots, \bbe_d\}$, where $\bbe_j \in \R^d$ is the standard basis element with a 1 in the $j$th coordinate and is zero elsewhere, and reports the original entry with probability slightly more than uniform and otherwise reports a different element.  Note that $\multRR$ takes a single data entry and is $\epsilon$-differentially private.

\begin{algorithm}
\caption{Exponential Mechanism: $\multRR$}
\label{alg:multRR}
\begin{algorithmic}
\REQUIRE $\bbx \in \{\bbe_1,\cdots, \bbe_d\}$, $\epsilon$.
\STATE Let $q(\bbx,\bbz) = \1\{\bbx =\bbz \}$
\STATE Select $\check {\bbx}$ with probability $\frac{\exp\left[\epsilon \ q(\bbx,\check{\bbx}) \right]}{e^\epsilon-1+d}$
\ENSURE $\check {\bbx}$
\end{algorithmic}
\end{algorithm}

We have the following result when we use $\multRR$ on each data entry to obtain a private histogram.
\begin{lemma}
If we have histogram $\bH = \sum_{i=1}^n\bX_i$, where $\{\bX_i\}
\stackrel{i.i.d.}{\sim} \mult(1,\bp)$ and we write $\check{\bH} =
\sum_{i=1}^n  \multRR(\bX_i,\epsilon)$ for each $i \in [n]$, then
$\check{\bH} \sim \mult(n,\check{\bp})$ where
\begin{equation}
\check{\bp} = \bp\left(\frac{e^\epsilon}{e^\epsilon + d-1}\right)  + (1-\bp)\left(\frac{1}{e^\epsilon+d-1}\right).
\label{eq:exphist}
\end{equation}
\label{lem:RRhist}
\end{lemma}

Once we have identified the distribution of $\check{\bH}$, we can create a chi-square statistic by subtracting $\check{\bH}$ by its expectation and dividing the difference by the expectation.  Hence testing $\nullhyp: \bp = \nullbp$ after the data has passed through the exponential mechanism, is equivalent to testing $\nullhyp: \bp = \check{\nullbp}$ with data $\check{\bH}$.  We will use the following classical result to prove our theorems.
\begin{theorem}[\cite{Ferg96}]
If $\bX \sim \Normal{\pmb{\mu}}{\Sigma}$ and $\Sigma$ is a projection matrix of rank $\nu$ and $\Sigma \pmb{\mu}  = \pmb{\mu}$ then $X^\intercal X \sim \chi^2_\nu ( \pmb{\mu}^\intercal \pmb{\mu})$.
\label{thm:classical}
\end{theorem}

We can then form a chi-square statistic using the private histogram $\check{\bH}$ which will have the correct asymptotic distribution.
\begin{theorem}
Let $\bH \sim \mult(n,\bp)$ and $\check{\bH}$ be given in \Cref{lem:RRhist} with privacy parameter $\epsilon>0$.  Under the null hypothesis $\nullhyp:\bp = \nullbp$, we have for $\check{\nullbp} = \frac{1}{e^\epsilon + d-1}\left( e^\epsilon \nullbp   + (1-\nullbp)\right)$, and
\begin{equation}
\chisqExp{\epsilon}  = \sum_{j=1}^d \frac{(\check{H}_j- n \check{p}_j^0 )^2}{n \check{p}_j^0} \stackrel{D}{\to} \chi_{d-1}^2.
\label{eq:localchisq}
\end{equation}
Further, with alternate $\althyp: \bp = \bp^1_n$, the resulting asymptotic distribution is the following,
$
\chisqExp{\epsilon}\stackrel{D}{\to} \chi_{d-1}^2\left( \left(\frac{e^\epsilon - 1}{e^\epsilon + d-1} \right)^2 \sum_{j=1}^d\frac{\Delta_j^2}{\check{p}_j^0} \right)
$
\label{thm:ExpChiTest}
\end{theorem}
 \begin{proof}
 Let $\bV = \sqrt{n} \left(\frac{\check{\bH}/n - \check{\nullbp} }{\check{\nullbp}} \right)$, which converges in distribution to $\Normal{\pmb{0}}{I_d - \sqrt{\check{\nullbp}} \sqrt{\check{\nullbp}}^\intercal}$, by the central limit theorem.  We then apply \Cref{thm:classical} with $\pmb{\mu} = \pmb{0}$ and $I_d - \sqrt{\check{\nullbp}} \sqrt{\check{\nullbp}}^\intercal$ being a projection matrix of rank $d-1$.  
 
 In the second statement we assume the alternate $\bp = \bp^1$ holds, in which case $\bV$ converges in distribution to  $\Normal{\left( \frac{e^\epsilon - 1}{e^\epsilon + d - 1} \right) \frac{\bDelta}{\check{\nullbp}}}{I_d - \sqrt{\check{\nullbp}} \sqrt{\check{\nullbp}}^\intercal}$.  We then verify that $\left( \frac{e^\epsilon - 1}{e^\epsilon + d - 1} \right)\left(I_d - \sqrt{\check{\nullbp}} \sqrt{\check{\nullbp}}^\intercal \right)\frac{\Delta}{\check{\nullbp}} =\left( \frac{e^\epsilon - 1}{e^\epsilon + d - 1} \right) \frac{\Delta}{\check{\nullbp}}$ to prove the second statement.
\end{proof}

We then base our LDP goodness of fit test on this result to obtain the correct critical value to reject the null hypothesis based on a chi-square distribution.  The test is presented in \Cref{alg:ExpGOF}. 
\begin{algorithm}
\caption{Local DP GOF Test: $\ExpGOF$}
\label{alg:ExpGOF}
\begin{algorithmic}
\REQUIRE $\bbx = (\bbx_1,\cdots, \bbx_n)$, $\epsilon$, $\alpha$, $\nullhyp: \bp = \nullbp$.
\STATE Let $\check{\bp}^0 = \frac{1}{e^\epsilon + d-1}\left( e^\epsilon \nullbp   + (1-\nullbp)\right)$.
\STATE Let $\check{\bH} = \sum_{i=1}^n \multRR(\bbx_i,\epsilon)$.  
\STATE Set $ q = \sum_{j=1}^d \frac{(\check{h}_j - n \check{p}_j^0)^2}{n \check{p}_j^0}$
\STATE {\bf if } $q > \chi_{d-1,1-\alpha}^2$
	Decision $\gets $ Reject.
\STATE {\bf else } 
	 Decision $\gets$ Fail to Reject.
\ENSURE Decision
\end{algorithmic}
\end{algorithm}
The following result is immediate from the exponential mechanism being $\epsilon$-DP and the fact that we use it as a local randomizer.
\begin{theorem}
\emph{$\ExpGOF$} is $\epsilon$-LDP.
\end{theorem}
\begin{proof}
The proof follows from the fact that we use $\multRR$ for each individual's data and then $\ExpGOF$ aggregates the privatized data, which is just a post-processing function.  
\end{proof}

\subsection{Goodness of Fit Test with Bit Flipping\label{sect:BitFlip}}
Note that the noncentral parameter in \Cref{thm:ExpChiTest} goes to zero as $d$ grows large due to the coefficient being $\left(\frac{e^\epsilon - 1}{e^\epsilon + d-1} \right)^2$.  Thus, for large dimensional data the exponential mechanism cannot reject a false null hypothesis.  We next consider another differentially private algorithm $\alg:\{\bbe_1,\cdots, \bbe_d \} \to \{0,1\}^d$, given in \Cref{alg:localJL} used in \cite{BS15} that flips each bit with some biased probability.


\begin{algorithm}
\caption{Bit Flip Local Randomizer: $\Mbit$}
\label{alg:localJL}
\begin{algorithmic}
\REQUIRE $\bbx \in \{\bbe_1,\cdots, \bbe_d\}$, $\epsilon$.
\FOR{$ j \in [d]$}
\STATE Set $z_j = x_j$ with probability $\frac{e^{\epsilon/2}}{e^{\epsilon/2}+1}$, otherwise $z_j = (1-x_j)$.  
\ENDFOR
\ENSURE $\bbz$
\end{algorithmic}
\end{algorithm}

\begin{theorem}
The algorithm $\Mbit$ is $\epsilon$-DP.  
\end{theorem}

We then want to form a statistic based on the output $\bbz \in \{0,1\}^d$  that is asymptotically distributed as a chi-square under the null hypothesis.  We defer the proof to the appendix.
\begin{lemma}
We consider data $\bX_i \sim \mult(1,\bp)$ for each $i \in [n]$.  We define the following covariance matrix $\covarJL(\bp)$ and mean vector $\tilde{\bp} = \frac{ \left[ \left(e^{\epsilon/2}-1\right)\bp +1 \right]}{e^{\epsilon/2}+1}$, in terms of $\alpha_\epsilon =  \left(\frac{e^{\epsilon/2}-1}{e^{\epsilon/2}+1} \right)$
\begin{align}
 \covarJL(\bp) = & \alpha_\epsilon^2 \left[ \diag\left( \bp \right) - \bp\left( \bp \right)^\transpose \right] +\frac{e^{\epsilon/2}}{\left(e^{\epsilon/2}+1\right)^2} I_d
\label{eq:covarBitFlip}
\end{align}
The histogram $\tilde{\bH} =  \sum_{i=1}^n \Mbit(\bX_i)$ has the following asymptotic distribution
$
\sqrt{n} \left( \frac{\tilde{\bH}}{n}  - \tilde{\bp}\right) \stackrel{D}{\to} \Normal{ \pmb{0} }{\covarJL(\bp)}.
$
Further, $\covarJL(\bp)$ is invertible for any $\epsilon > 0$ and $\bp >\pmb{0}$.
\label{lem:technical}
\end{lemma}

Following a similar analysis in \cite{KR17} and using \Cref{thm:classical}, we can form the following statistic for null hypothesis $\nullhyp: \bp  =\nullbp$ in terms of the histogram $\tilde{\bH}$ and projection matrix $\Pi =  I_d - \frac{1}{d} \one \one^\transpose$, as well as the covariance matrix $\covarJL\left(\nullbp\right)$ and mean vector $\tilde{\bp}^0$ both given in \eqref{eq:covarBitFlip} where we replace $\bp$ with $\nullbp$:
\begin{equation}
\chisqBF{\epsilon} = n \left( \frac{\tilde{\bH}}{n} - \tilde{\bp}^0 \right)^\transpose\Pi \covarJL\left(\nullbp\right)^{-1}\Pi  \left( \frac{\tilde{\bH}}{n} - \tilde{\nullbp}]\right)
\label{eq:chisqbitflip}
\end{equation}
We can then design a hypothesis test based on the outputs from $\Mbit$ in \Cref{alg:BitGOF}
\begin{algorithm}
\caption{Local DP GOF Test: $\BitGOF$}
\label{alg:BitGOF}
\begin{algorithmic}
\REQUIRE $\bbx = (\bbx_1,\cdots, \bbx_n)$, $\epsilon$, $\alpha$, $\nullhyp: \bp = \nullbp$.
\STATE Let $\tilde{\bH} = \sum_{i=1}^n \Mbit(\bbx_i,\epsilon)$.  
\STATE Set $ q = \chisqBF{\epsilon}$
\STATE {\bf if } $q > \chi_{d-1,1-\alpha}^2$
	 Decision $\gets $ Reject.
\STATE {\bf else }
	 Decision $\gets$ Fail to Reject.
\ENSURE Decision
\end{algorithmic}
\end{algorithm}

\begin{theorem}
\emph{$\BitGOF$} is $\epsilon$-LDP.
\end{theorem}

We now show that the statistic in \eqref{eq:chisqbitflip} is
asymptotically distributed as $\chi_{d-1}^2$. 

\begin{theorem}
If the null hypothesis $\nullhyp:\bp = \nullbp$ holds, then the statistic $\chisqBF{\epsilon}$ 
is asymptotically distributed as a chi-square, i.e. $\chisqBF{\epsilon}\stackrel{D}{\to}\chi^2_{d-1}.$  Further, if we consider the alternate $\althyp: \bp = \bp^1$  then $\chisqBF{\epsilon}\stackrel{D}{\to}\chi^2_{d-1}\left(\left(\frac{e^{\epsilon/2} -1}{e^{\epsilon/2}+1}\right)^2 \cdot \bDelta^\transpose\covarJL(\nullbp)^{-1}\bDelta \right).$ 
\end{theorem}
\begin{proof}
From \Cref{lem:technical}, we know that $\sqrt{n} \left( \frac{\tilde{\bH}}{n} - \tilde{\nullbp}  \right) \stackrel{D}{\to} \Normal{\pmb{0}}{\Sigma(\tilde{\nullbp})}$.  Since $\Sigma(\tilde{\nullbp})$ is full rank and has an eigenvector that is all ones.  Thus, we can diagonalize the covariance matrix as $\Sigma(\tilde{\nullbp}) = B D B^\intercal$ where $D$ is a diagonal matrix and $B$ has orthogonal columns with one of them being $\frac{1}{d} \pmb{1}$.  Thus, we have 
$$
\sqrt{n} \Sigma(\tilde{\nullbp})^{-1/2} \Pi  \left( \frac{\tilde{\bH}}{n} - \tilde{\nullbp}  \right) \stackrel{D}{\to} \Normal{\pmb{0}}{\Lambda}
$$
where $\Lambda = \Sigma(\tilde{\nullbp})^{-1/2} \Pi BDB^\intercal \Pi \Sigma(\tilde{\nullbp})^{-1/2}$.  Hence, $\Lambda$ will be the identity matrix with a single zero on the diagonal.  We then apply \Cref{thm:classical} tp prove the first statement.

To prove the second statement we assume that $\bp = \bp^1$ holds, which gives $\sqrt{n} \left( \frac{\tilde{\bH}}{n} - \tilde{\nullbp}  \right) \stackrel{D}{\to} \Normal{ \alpha_\epsilon \cdot \bDelta }{\Sigma(\tilde{\nullbp})}$.  Thus, 
$$
\sqrt{n} \Sigma(\tilde{\nullbp})^{-1/2} \Pi  \left( \frac{\tilde{\bH}}{n} - \tilde{\nullbp}  \right) \stackrel{D}{\to} \Normal{ \alpha_\epsilon \cdot \Sigma(\tilde{\nullbp})^{-1/2}\bDelta }{\Lambda}
$$
We again use \Cref{thm:classical} to obtain the noncentral parameter $\alpha_\epsilon^2 \cdot \bDelta^\intercal \Sigma(\tilde{\nullbp} )^{-1} \bDelta$.
\end{proof}

\subsection{Comparison of Noncentral Parameters}

We now compare the noncentral parameters of the three local private tests we presented in \Cref{alg:LocalGaussGOF,alg:ExpGOF,alg:BitGOF}.  We consider the null hypothesis $\nullbp = (1/d,\cdots, 1/d)$ for $d > 2$, and alternate $\althyp:\bp = \nullbp + \frac{\bDelta}{\sqrt{n}}$.  In this case, we can easily compare the various noncentral parameters for various privacy parameters and dimensions $d$.  In \Cref{fig:noncentrals} we give the coefficient to the term $\bDelta^\transpose \bDelta$ in the noncentral parameter of the asymptotic distribution for each local private test presented thus far.  The larger this coefficient is, the better the power will be for any alternate $\bDelta$ vector.  Note that in $\GaussGOF$, we set $\rho = \epsilon^2/8$ which makes the variance the same as for a random variable distributed as $\Lap(2/\epsilon)$ for an $\epsilon$-DP guarantee -- recall that $\GaussGOF$ with Gaussian noise does not satisfy $\epsilon$-DP for any $\epsilon>0$.  We give results for $\epsilon \in \{1,2,3,4\}$ which are all in the range of privacy parameters that have been considered in actual locally differentially private algorithms used in practice.\footnote{In \cite{EPK14}, we know that Google uses $\epsilon = \ln(3) \approx 1.1$ in RAPPOR and from Aleksandra Korolova's Twitter post on September 13, 2016 \url{https://twitter.com/korolova/status/775801259504734208} we know that Apple uses $\epsilon = 1,4$.  }  From the plots, we see how  $\ExpGOF$ may outperform $\BitGOF$ depending on the privacy parameter and dimension of the data.   We can use these plots to determine which test to use given $\epsilon$ and dimension of data $d$.  When $\nullhyp$ is not uniform, we can use the noncentral parameters given for each test to determine which test has the largest noncentral parameter for a particular privacy budget $\epsilon$.

\begin{figure}[t!]
\begin{subfigure}{.245\textwidth}
  \centering
  \includegraphics[width=\linewidth]{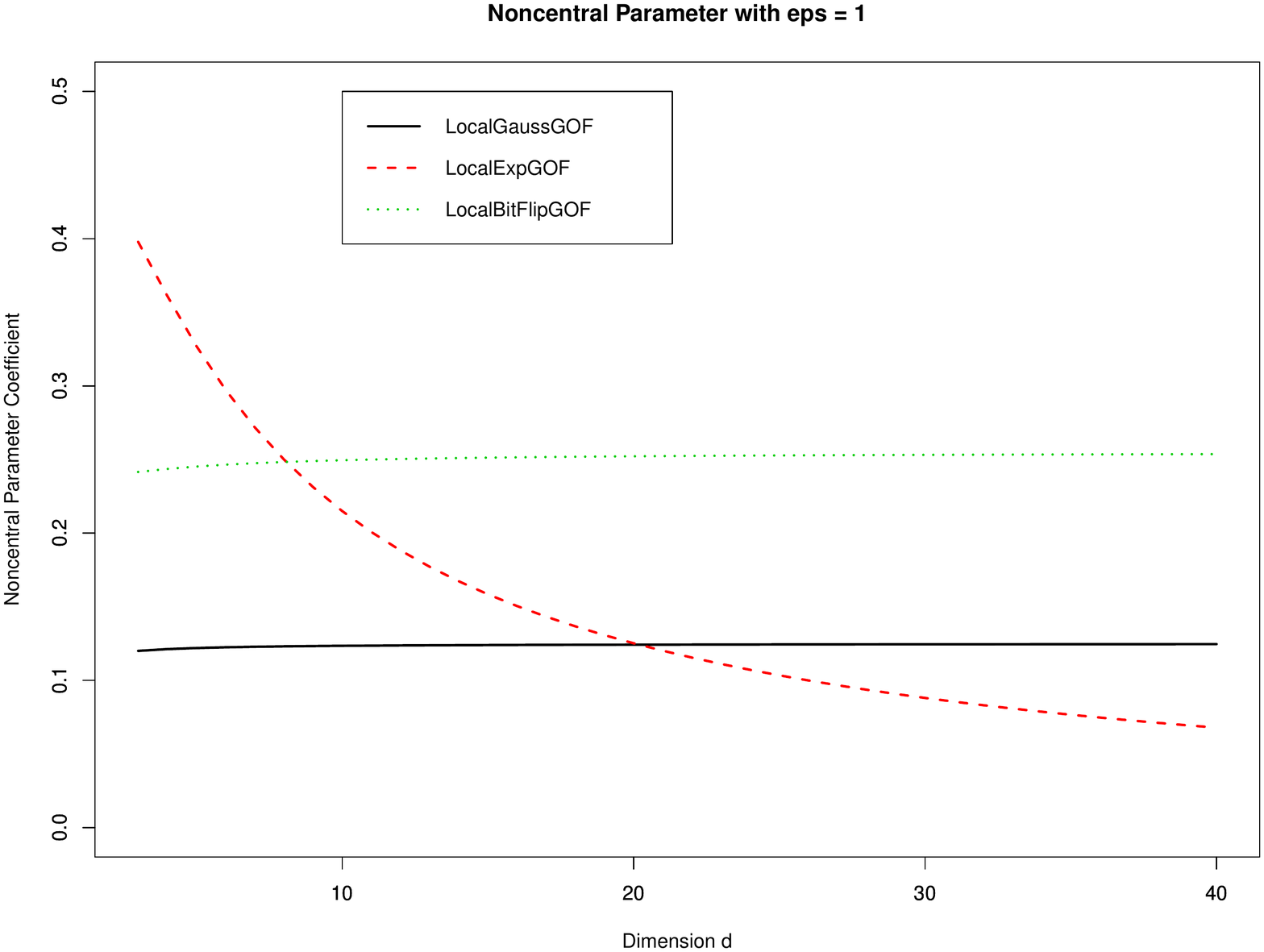}
\end{subfigure}%
\hfill
\begin{subfigure}{.245\textwidth}
  \centering
    \includegraphics[width=\linewidth]{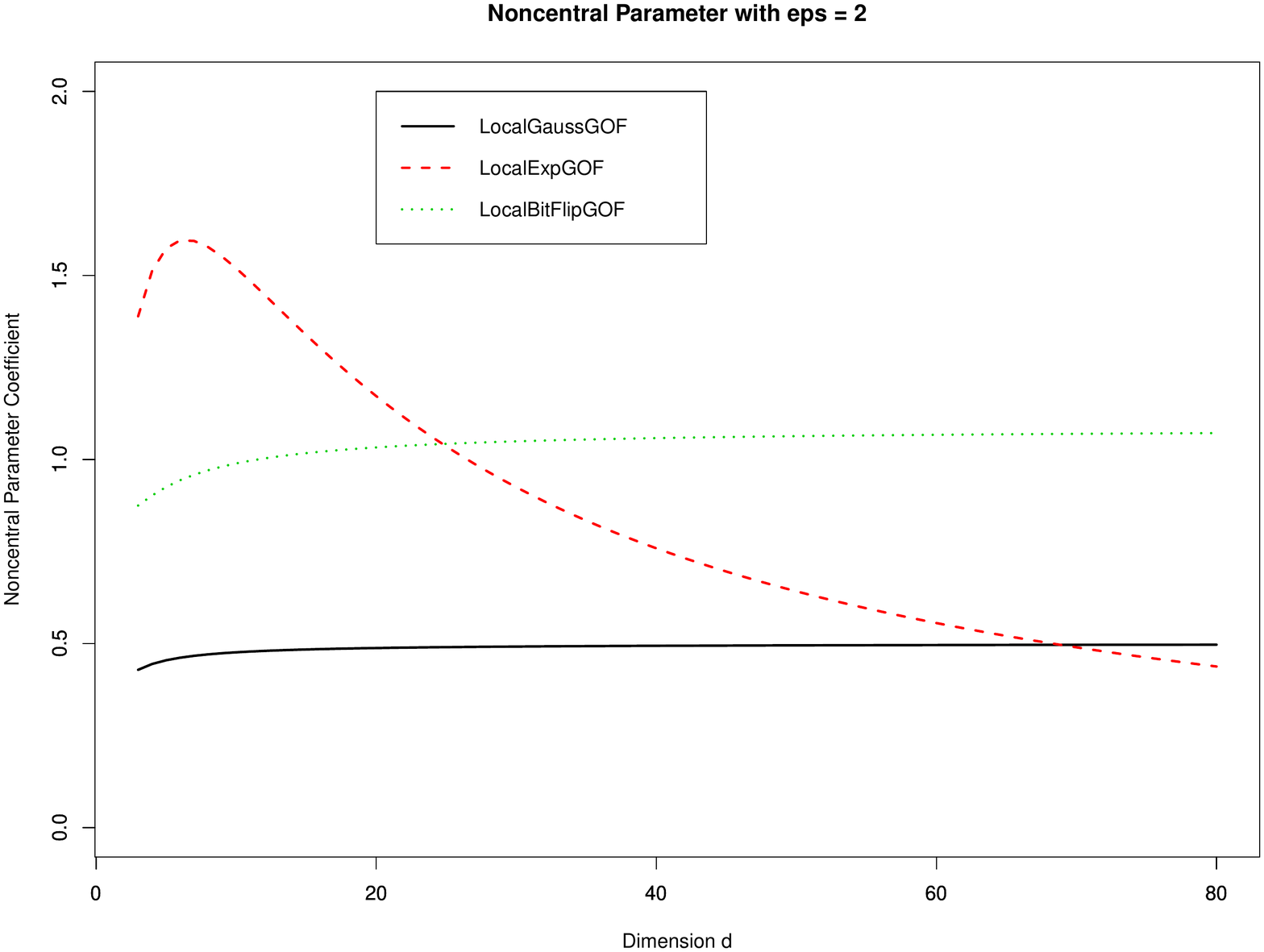}
\end{subfigure}
\begin{subfigure}{.245\textwidth}
  \centering
  \includegraphics[width=\linewidth]{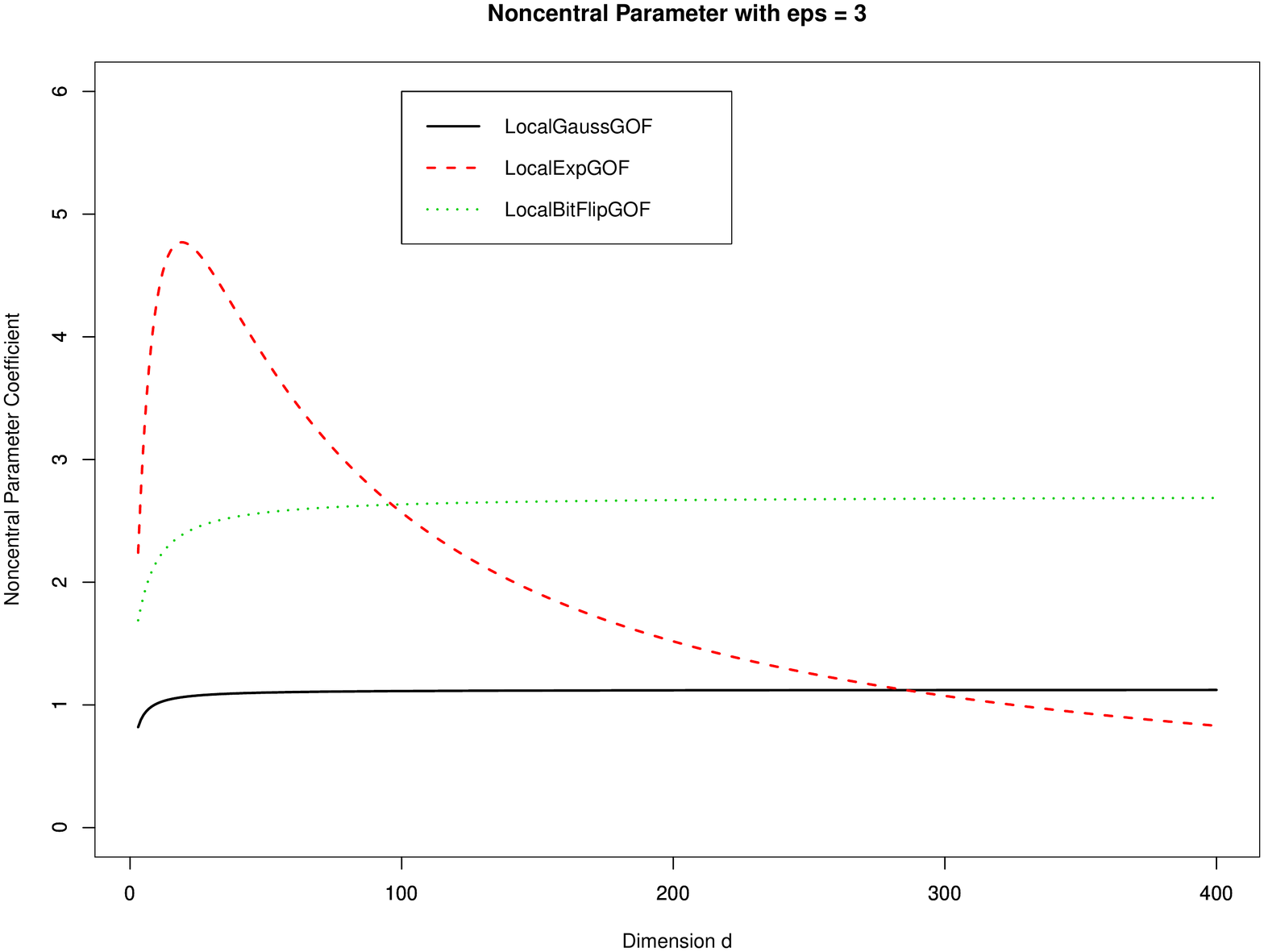}
\end{subfigure}%
\begin{subfigure}{.245\textwidth}
  \centering
  \includegraphics[width=\linewidth]{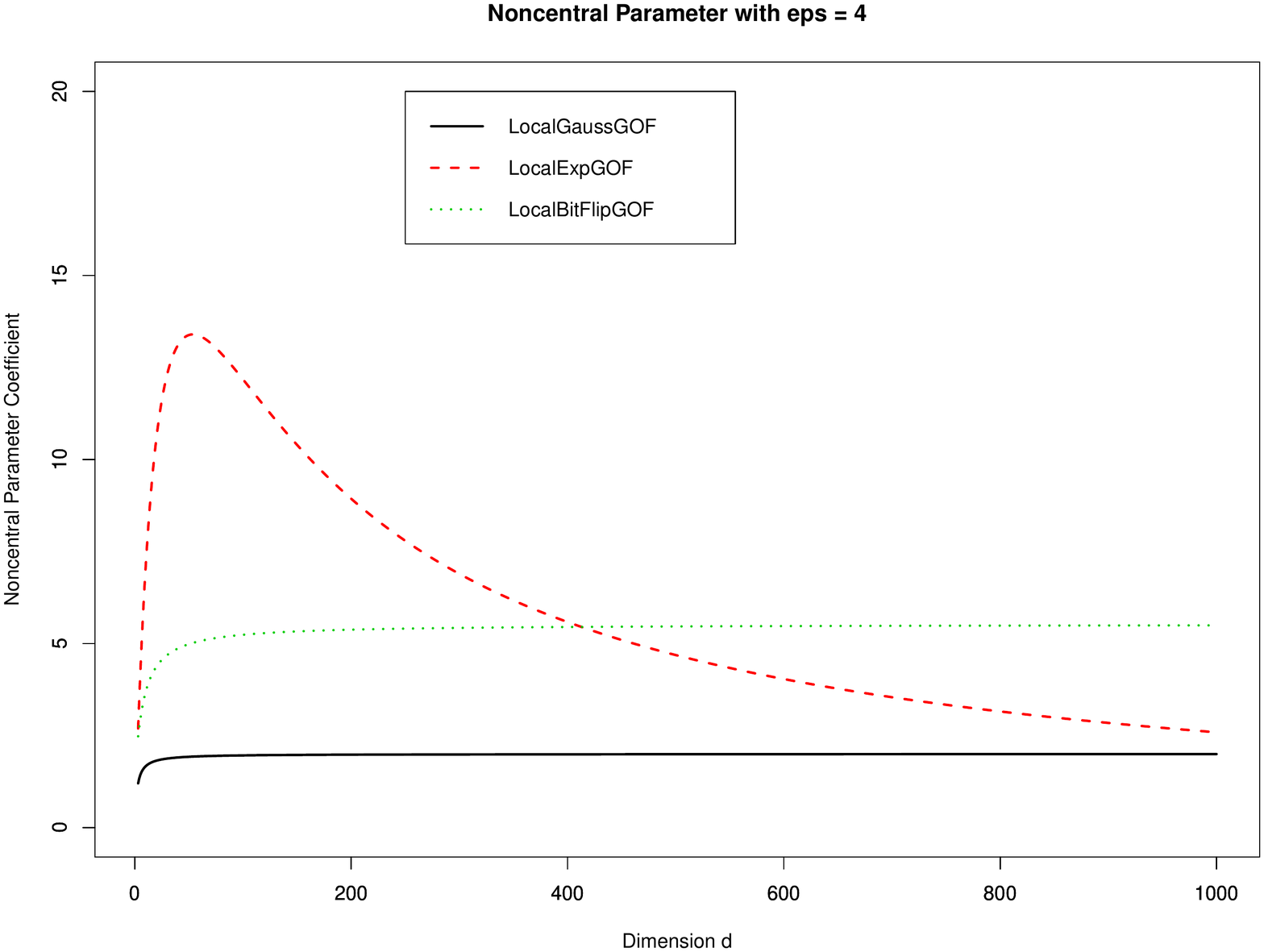}
\end{subfigure}%
\caption{The coefficient on $\bDelta^\transpose \bDelta$ in the noncentral parameter for the local private tests where ``LocalGaussGOF" is $\GaussGOF$ with Gaussian noise, $\ExpGOF$, and $\BitGOF$ for various dimensions $d$ and $\epsilon$.}
\label{fig:noncentrals}
 \end{figure}

\subsection{Empirical Results}

We then empirically compare the power between $\GaussGOF$ with Laplace noise in \Cref{alg:LocalGaussGOF}, $\ExpGOF$ in \Cref{alg:ExpGOF}, and $\BitGOF$ in \Cref{alg:BitGOF}.  Recall that all three of these tests have the same privacy benchmark of local differential privacy.  For $\GaussGOF$ with Laplace noise, we will use $m = 999$ samples in out Monte Carlo simulations.  In our experiments we fix $\alpha = 0.05$ and $\epsilon \in \{1,2,4\}$.  We then consider null hypotheses of the form $\nullbp = (1/d,1/d, \cdots, 1/d)$ and alternate $\althyp: \bp = \nullbp + \eta (1,-1,\cdots, 1,-1)$ for some $\eta>0$.  In \Cref{fig:LocalDPGOFPower}, we plot the number of times our tests correctly rejects the null hypothesis in 1000 independent trials for various sample sizes $n$ and privacy parameters $\epsilon$.
\begin{figure}[h]
\begin{subfigure}{.49\textwidth}
  \centering
\includegraphics[width=.9\textwidth]{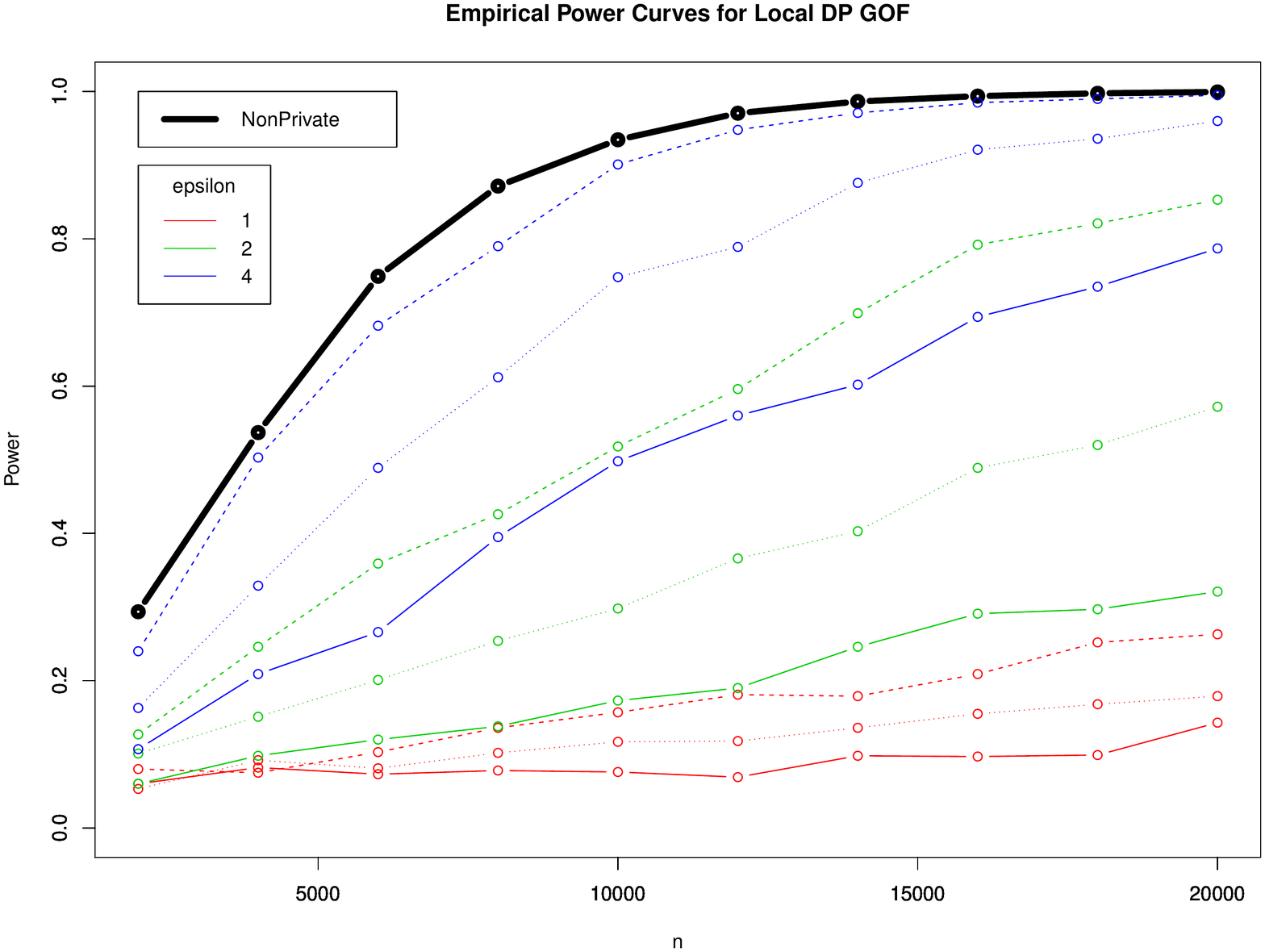}
\end{subfigure}
\hfill
\begin{subfigure}{.49\textwidth}
  \centering
\includegraphics[width=.9\textwidth]{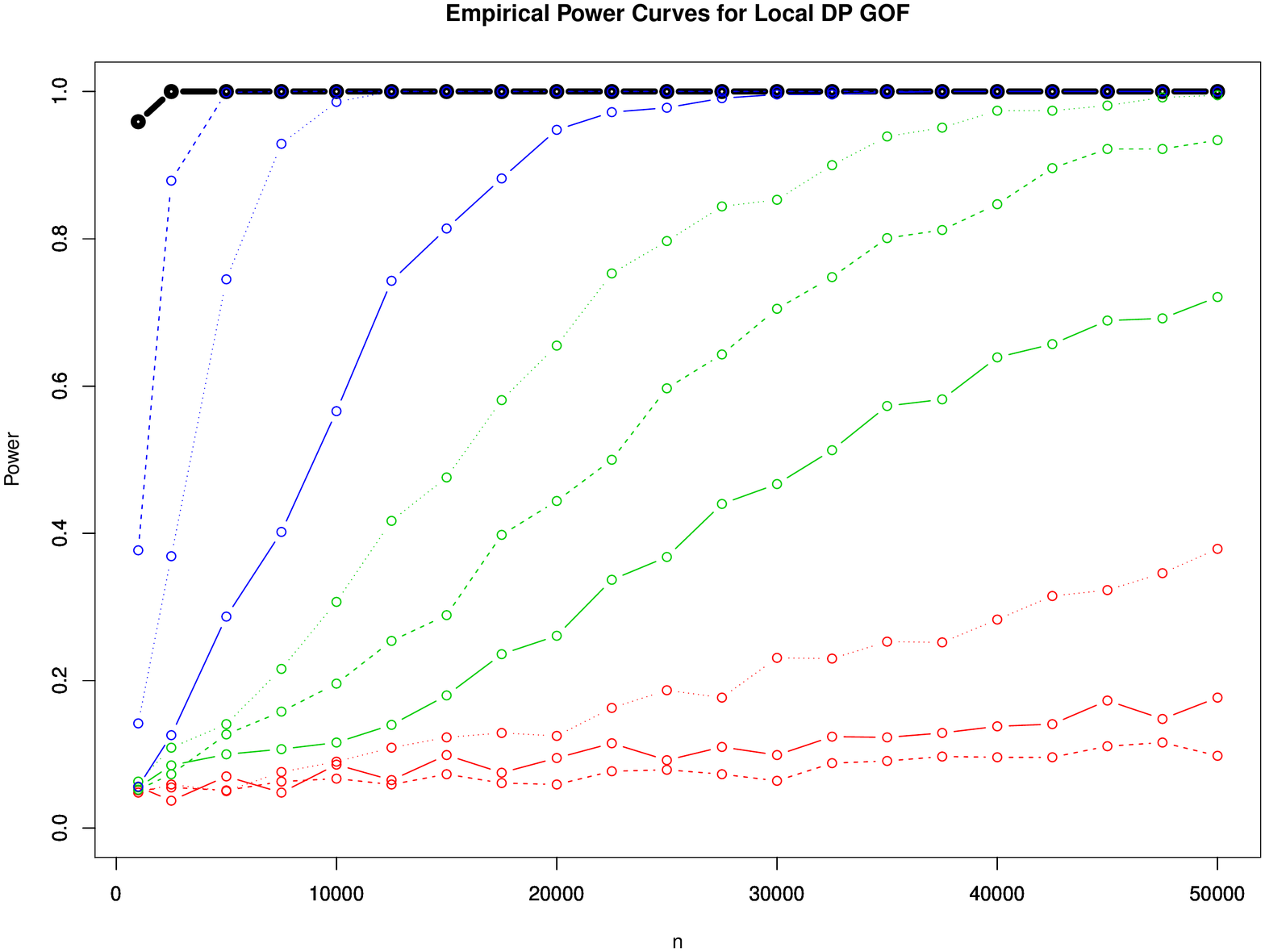}
\end{subfigure}
\caption{Comparison of empirical power among the classical non-private test and the local private tests $\GaussGOF$ with Laplace noise (solid line), $\ExpGOF$ (dashed line), and $\BitGOF$ (dotted line);  in the left plot, $d = 4$ and $\eta = 0.01$, in the right plot, $d = 40$ and $\eta = 0.005$.
\label{fig:LocalDPGOFPower}}
\end{figure}

From \Cref{fig:LocalDPGOFPower}, we can see that the test statistics that have the largest noncentral parameter for a particular dimension $d$ and privacy parameter $\epsilon$ will have the best empirical power.  When $d = 4$, we see that $\ExpGOF$ performs the best.  However, for $d = 40$ it is not so clear cut.  When $\epsilon = 4$, we can see that $\ExpGOF$ does the best, but then when $\epsilon = 2$, $\BitGOF$ does best.  Thus, the best Local DP Goodness of Fit test depends on the noncentral parameter, which is a function of $\epsilon$, the null hypothesis $\nullbp$, and alternate $\bp = \nullbp + \bDelta$.  Note that the worst local DP test also depends on the privacy parameter and the dimension $d$.  
Based on our empirical results, we see that no single locally private
test is best for all data dimensions.  Knowing the corresponding
noncentral parameter for a given problem is useful in determining
which tests to use, where the larger the noncentral parameter is the
higher the power will be.  
\section{Local Private Independence Tests}

We now show that our techniques can be extended to include \emph{composite} hypothesis tests, where we test whether the data comes from a whole family of probability distributions.  We will focus on independence testing, but much of the theory can be extended to general chi-square tests.  We will closely follow the presentation and notation as in \citep{KR17}.  

We consider two multinomial random variables $\{ \bU_\ell \}^n_{\ell = 1} \stackrel{i.i.d.}{ \sim} \mult(1,\bpipi{1})$ for $\bpipi{1} \in \R^{r}$, $\{\bV_\ell \}_{\ell = 1}^n \stackrel{i.i.d.}{ \sim} \mult(1,\bpipi{2})$ for $\bpipi{2} \in \R^{c}$ and no component of $\bpipi{1}$ or $\bpipi{2}$ is zero. Without loss of generality, we will consider an individual to be in one of $r$ groups who reports a data record that is in one of $c$ categories.  The collected data consists of $n$ joint outcomes $\bH$ whose $(i,j)$th coordinate is $H_{i,j} =  \sum_{\ell = 1}^n \1\{U_{\ell,i} = 1 \quad \& \quad V_{\ell,j} = 1\} $.  Note that $\bH$ is then the contingency table over the joint outcomes.  

Under the null hypothesis of independence between $\{ \bU_\ell \}^n_{\ell = 1} $ and $\{ \bV_\ell \}^n_{\ell = 1}$, we have 
\begin{equation}
\bH \sim \mult\left(n,\bp(\bpipi{1},\bpipi{2})\right)\qquad  \text{ where } \bp(\bpipi{1}, \bpipi{2}) = \bpipi{1} \left(\bpipi{2}\right)^\intercal 
\label{eq:privhist}
\end{equation}
What makes this test difficult is that the analyst does not know the data distribution $\bp(\bpipi{1},\bpipi{2})$ and so cannot simply plug it into the chi-square statistic.  Rather, we use the data to estimate the \emph{best guess} for the unknown probability distribution that satisfies the null hypothesis.  

Note that without privacy, each individual $\ell \in [n]$ is reporting a $r \times c$ matrix $\bX_\ell$ which would be 1 in exactly one location.  Thus we can alternatively write the contingency table as $\bH = \sum_{\ell = 1}^n \bX_\ell$.  We then use the three local private algorithms we presented earlier for noise addition, exponential mechanism, and bit flipping to see how we can form a private chi-square statistic for independence testing.  We want to be able to ensure the privacy of both the group and the category that  each individual belongs to.  

For completeness, we will present the minimum chi-square asymptotic theory from \cite{Ferg96} that was used in \cite{KR17}.  We will use this theory to design general chi-square test statistics for local differentially private hypothesis tests.  Consider a $d$ dimensional random vector $\bV^{(n)}$ and a $k \leq d$ dimensional parameter space $\Theta$, which is an open subset of $\R^k$.  We want to measure the distance between the random variable $\bV^{(n)}$ and some model $A(\btheta)$ that maps parameters in $\Theta$ to $\R^d$.  We then use the following quadratic form where $M(\btheta) \in \R^{d\times d}$ is a positive-semidefinite matrix.  
\begin{equation}
n \left(\bV^{(n)} - A(\btheta) \right)^\intercal M(\btheta) \left(\bV^{(n)} - A(\btheta) \right)
\label{eq:quadratic}
\end{equation}
We want to ensure the following assumptions hold.

\begin{assumpt}
For all $\btheta \in \Theta$, we have
\begin{itemize}
\item There exists a parameter $\btheta^\star \in \Theta$ such that for some model $A(\cdot)$ and covariance matrix $C(\cdot)$
$$
\sqrt{n} \left( \bV^{(n)} - A(\btheta^\star) \right) \stackrel{D}{\to} \Normal{\pmb{0}}{C(\btheta^\star)}
$$
\item $A(\btheta)$ is bicontinuous.
\item $A(\btheta)$ has continuous first partial derivatives, denoted as $\dot{A}(\btheta)$, with full rank $k$.
\item The matrix $M(\btheta)$ is continuous in $\btheta$ and there exists an $\eta >0$ such that $M(\btheta) - \eta I_d$ is positive definite in an open neighborhood of $\btheta^\star$.  
\end{itemize}
\label{assumpt:one}
\end{assumpt}
We will then use the following useful result from \cite{KR17}[Theorem 4.2].
\begin{theorem}[\citet{KR17}]
Let $\nu$ be the rank of $C(\btheta^\star)$.  If \Cref{assumpt:one} holds, and for all $\btheta \in \Theta$ we have 
$$
C(\btheta) M(\btheta) C(\btheta) = C(\btheta)
$$
$$
C(\btheta) M(\btheta) \dot{A}(\btheta) = \dot{A}(\btheta) 
$$
Then for any estimate $\phi(\bV^{(n)}) \stackrel{P}{\to} \btheta^\star$ we have 
$$
\min_{\btheta \in \Theta} \left\{ n \left(\bV^{(n)} - A(\btheta) \right)^\intercal M(\phi(\bV^{(n)})) \left(\bV^{(n)} - A(\btheta) \right)\right\} \stackrel{D}{\to} \chi^2_{\nu-k}.
$$
\label{thm:KRGeneral}
\end{theorem}

To ensure the assumptions hold, we will assume that each component of $\bpipi{1}, \bpipi{2}$ is positive, i.e. $\bpipi{i} > 0$ for $i = 1,2$.

\subsection{Independence Test with Noise Addition}
For noise addition, we will have the following contingency table $\bH + \bZ$ (treating $\bH$ as an $r \times c$ vector) where $\bZ = \Normal{\pmb{0}}{\frac{n}{\rho} }$ for $\rho$-LzCDP or $\bZ$ is the sum of $n$ independent Laplace random vectors with scale parameter $2/\epsilon$ in each coordinate for $\epsilon$-LDP, although we will only consider the Gaussian noise case.  We can then use the same statistic presented in \citet{KR17} with a rescaling of the variance in the noise and flattening the contingency table as a vector.  For convenience, we write the marginals as $H_{i,\cdot} = \sum_{j = 1}^c H_{i,j}$ and similarly for $H_{\cdot, j}$.  Further, we will write $\hat{n} = n + \sum_{i,j} Z_{i,j}$ in the following test statistic,
\footnote{We will follow a common rule of thumb for small sample sizes, so that if $n \hbpipi{1} \left(\hbpipi{2}\right)^\intercal \leq 5$ for any cell, then we simply \emph{fail to reject} $\nullhyp$ }
\begin{align}
\hbpipi{1} & = \left( \frac{1}{\hat{n}}\left(H_{i,\cdot} + Z_{i,\cdot} \right) : i \in [r]\right), \qquad \hbpipi{2} =  \left( \frac{1}{\hat{n}} \left(H_{\cdot,j} + Z_{\cdot,j} \right) : j \in [c]\right) \nonumber \\
\hat{M}& =  \diag\left( \bp(\hbpipi{1}, \hbpipi{2} )\right) -  \bp(\hbpipi{1}, \hbpipi{2} ) \left( \bp(\hbpipi{1}, \hbpipi{2} )\right)^\intercal + \frac{1}{\rho} I_{r c}\nonumber \\
\bchisqKR{\rho/n; (\bpipi{1},\bpipi{2})} & = \frac{1}{n} \left( \bH + \bZ -  n \bp\left(\bpipi{1}, \bpipi{2}\right) \right)^\transpose \Pi \left(\hat{M} \right)^{-1} \Pi \left(  \bH + \bZ - n \bp\left(\bpipi{1}, \bpipi{2}\right)  \right)
\label{eq:KRind}
\end{align}

We use this statistic because it is asymptotically distributed as a chi-square random variable given the null hypothesis holds, which follows directly from the general chi-square theory presented in \cite{KR17}.
\begin{theorem}
Under the null hypothesis that $\bU$ and $\bV$ are independent, then we have as $n \to \infty$
$$
\min_{\bpipi{1},\bpipi{2}} \left\{ \bchisqKR{\rho/n; (\bpipi{1},\bpipi{2})} \right\} \stackrel{D}{\to} \chi^2_{(r-1)(c-1)}.
$$ 
\end{theorem}

We present the test in \Cref{alg:LocalGaussIND} for Gaussian noise which uses the statistic in \eqref{eq:KRind}.

\begin{algorithm}
\caption{LzCDP Independence Test with Gaussian Mechanism: $\GaussIND$}
\label{alg:LocalGaussIND}
\begin{algorithmic}
\REQUIRE $(\bbx_1,\cdots, \bbx_n)$, $\rho$, $\alpha$.
\STATE Let $\bH = \sum_{\ell = 1}^n \bbx_\ell$ 
\STATE Set $ q = \min_{\bpipi{1},\bpipi{2}} \left\{ \bchisqKR{\rho/n; (\bpipi{1},\bpipi{2})} \right\}$ given in \eqref{eq:KRind}.
\IF{$q > \chi_{(r-1)(c-1),1-\alpha}^2$}
	\STATE Decision $\gets $ Reject.
\ELSE
	\STATE Decision $\gets$ Fail to Reject.
\ENDIF
\ENSURE Decision
\end{algorithmic}
\end{algorithm}

We then have the following result which follows from the privacy analysis from before.
\begin{theorem}
$\GaussIND$ is $\rho$-LzCDP.  
\end{theorem}

\subsection{Independence Test with the Exponential Mechanism}
Next we want to design an independence test when the data is generated from $\multRR$ given in \Cref{alg:multRR}.  In this case our contingency table can be written as $\check{\bH} \sim \mult\left(n, \check{\bp}(\bpipi{1},\bpipi{2}) \right)$ where $\beta_\epsilon = \frac{1}{e^\epsilon + rc-1}$ and we use  \eqref{eq:exphist} to get 
\begin{equation}
\check{\bp}(\bpipi{1},\bpipi{2}) = \beta_\epsilon \left(\left(e^\epsilon -1 \right) \bpipi{1}\left(\bpipi{2}\right)^\intercal \   +\pmb{1} \right)
\label{eq:probexpind}
\end{equation}
We then obtain an estimate for the unknown parameters,

\begin{align}
\cbpipi{1} & =\frac{1}{\beta_\epsilon\left(e^\epsilon - 1\right)} \left( \frac{\check{H}_{i,\cdot}}{n}   - c \beta_\epsilon : i \in [r]\right), \nonumber\\ 
\cbpipi{2} & = \frac{1}{\beta_\epsilon\left(e^\epsilon - 1\right)} \left( \frac{\check{H}_{\cdot,j}}{n}  - r \beta_\epsilon : j \in [c]\right) \nonumber \\
\bchisqExp{\epsilon} & = \sum_{i,j} \frac{\left(\check{H}_{i,j}- n \check{p}_{i,j}\left( \cbpipi{1},\cbpipi{2} \right) \right)^2}{n \check{p}_{i,j}(\cbpipi{1},\cbpipi{2}) } 
\label{eq:Expind}
\end{align}

We can then prove the following result, which uses \Cref{thm:KRGeneral}.  
\begin{theorem}
Assuming $\bU$ and $\bV$ are independent with true probability vectors $\bpipi{1}, \bpipi{2} > 0$ respectively, then as $n \to \infty$ we have
$
\bchisqExp{\epsilon} \stackrel{D}{\to} \chi^2_{(r-1)(c-1)}
$. 
\label{thm:ExpInd}
\end{theorem}
\begin{proof}
Note that $\cbpipi{1}, \cbpipi{2}$ both converge in probability to the true parameters $\bpipi{1}, \bpipi{2}$, respectively due to $\ex{\check{\bH}/n}  = \check{\bp}(\bpipi{1},\bpipi{2})$ and $\var{\check{\bH}/n} \to 0$.  We then form the following statistic in terms of parameters $\btheta^{(1)}$ in the $r$ dimensional simplex with all positive components and $\btheta^{(2)}$ in the $c$ dimensional simplex with all positive components,
\begin{align*}
& \bchisqExp{\btheta^{(1)}, \btheta^{(2)} ; \epsilon}\\
& \quad = n \left( \check{\bH}/n - \check{\bp}(\btheta^{(1)}, \btheta^{(2)}) \right)^ \intercal \diag\left(\check{\bp}(\cbpipi{1}, \cbpipi{2}) \right)^{-1} \left( \check{\bH}/n - \check{\bp}(\btheta^{(1)}, \btheta^{(2)})  \right)
\end{align*}
We then have the following calculation,
$$
\left( \cbpipi{1}, \cbpipi{2} \right) = \myargmin_{\btheta^{(1)}, \btheta^{(2)}}\left\{\bchisqExp{\btheta^{(1)}, \btheta^{(2)} ; \epsilon} \right\}.
$$
Note that $\sqrt{n} \left( \check{\bH}/n - \check{\bp}(\bpipi{1},\bpipi{2}) \right) \stackrel{D}{\to} \Normal{\pmb{0}}{C(\btheta)}$
where $C(\btheta) =  \diag( \check{\bp}(\btheta^{(1)}, \btheta^{(2)})) -  \check{\bp}(\btheta^{(1)}, \btheta^{(2)}) \left( \check{\bp}(\btheta^{(1)}, \btheta^{(2)}) \right)^\intercal$ is the covariance matrix for the multinomial random variable.

We then apply \Cref{thm:KRGeneral} to prove the statement, so we verify that the following equalities hold,
\begin{align*}
C(\btheta) \cdot \diag\left( \check{\bp}(\btheta^{(1)}, \btheta^{(2)}) \right)^{-1}  \cdot C(\btheta) =  C(\btheta)
\end{align*}
and 
\begin{align*}
C(\btheta)&  \cdot \diag\left( \check{\bp}(\btheta^{(1)}, \btheta^{(2)}) \right)^{-1} \nabla\left( \btheta^{(1)} \left(\btheta^{(2)}\right)^\intercal \right) \\
& = \left( I_{rc} -  \check{\bp}(\btheta^{(1)}, \btheta^{(2)}) \pmb{1}^\intercal  \right) \nabla\left( \btheta^{(1)} \left(\btheta^{(2)}\right)^\intercal \right)\\
& =  \nabla\left( \btheta^{(1)} \left(\btheta^{(2)}\right)^\intercal \right)
\end{align*}
where the last equality follows from the last coordinates of each probability vector can be written as $\theta^{(1)}_r = 1- \sum_{i=1}^{r-1}\theta_i$ and $\theta^{(2)}_c = 1 - \sum_{j=1}^{c-1}\theta_j^{(2)}$.
\end{proof}

We then use this result to design our private chi-square test for independence.
\begin{algorithm}
\caption{Local DP IND Test: $\ExpIND$}
\label{alg:LocalExpIND}
\begin{algorithmic}
\REQUIRE $\bbx = (\bbx_1,\cdots, \bbx_n)$, $\epsilon$, $\alpha$, $\nullhyp: \bp = \nullbp$.
\STATE Let $\check{\bH} = \sum_{i=1}^n \multRR(\bbx_i,\epsilon)$.  
\STATE Set $ q =\bchisqExp{\epsilon}$ from \eqref{eq:Expind}
\STATE {\bf if } $q > \chi_{d-1,1-\alpha}^2,$
	 Decision $\gets $ Reject.
\STATE {\bf else }
	 Decision $\gets$ Fail to Reject.
\ENSURE Decision
\end{algorithmic}
\end{algorithm}

We then have the following result which again follows from the privacy analysis from before.
\begin{theorem}
$\ExpIND$ is $\epsilon$-LDP.  
\end{theorem}
\subsection{Independence Test with Bit Flipping}
Lastly, we design an independence test when the data is reported via $\Mbit$ in \Cref{alg:localJL}.  Assuming that $\bH = \sum_{\ell=1}^n\bX_\ell \sim \mult\left(n, \bp(\bpipi{1}, \bpipi{2}) \right)$, then we know that replacing $\nullbp$ with $\bp(\bpipi{1},\bpipi{2})$ in \Cref{sect:BitFlip} gives us the following asymptotic distribution (treating the contingency table of values as a vector) with covariance matrix $\covarJL(\cdot)$ given in \eqref{eq:covarBitFlip}
\begin{align}
& \sqrt{n} \left( \frac{\tilde{\bH}}{n}  - \left[ \underbrace{ \left( \frac{e^{\epsilon/2} -1 }{e^{\epsilon/2} +1 } \right) \bpipi{1}\left(\bpipi{2}\right)^\intercal + \frac{1}{e^{\epsilon/2} + 1}}_{\tilde{\bp}(\bpipi{1},\bpipi{2})}\right] \right) \nonumber\\
& \qquad \stackrel{D}{\to} \Normal{ \pmb{0} }{\covarJL\left( \bpipi{1} \left(\bpipi{2} \right)^\transpose \right)}
\label{eq:BF_CLT}
\end{align}
Similar to analysis for \Cref{thm:ExpInd}, we start with a rough estimate for the unknown parameters which converges in probability to the true estimates, so we again use $\alpha_\epsilon =  \left(\frac{e^{\epsilon/2}-1}{e^{\epsilon/2}+1} \right)$ to get
\begin{align}
\tbpipi{1} &= \left( \frac{1}{\alpha_\epsilon} \right) \left( \frac{\tilde{H}_{i, \cdot}}{n} - \frac{c}{e^{\epsilon/2} + 1} : i \in [r]\right) \nonumber \\
\tbpipi{2}& = \left( \frac{1}{\alpha_\epsilon} \right) \left( \frac{\tilde{H}_{\cdot, j}}{n} - \frac{r}{e^{\epsilon/2} + 1} : j \in [c]\right)
\end{align}

We then give the resulting statistic, parameterized by the unknown parameters $\bpipi{\ell}$, for $\ell \in \{1,2\}$.  For middle matrix $\tilde{M}  = \Pi \covarJL\left( \tbpipi{1} \left(\tbpipi{2} \right)^\transpose \right) ^{-1} \Pi$, we have
\begin{align}
&\bchisqBF{\btheta^{(1)},\btheta^{(2)};\epsilon}     = \frac{1}{n} \left( \tilde{\bH} -  n \tilde{\bp}\left(\btheta^{(1)},\btheta^{(2)}\right) \right)^\transpose \nonumber   \\
& \qquad \tilde{M} \left(  \tilde{\bH} - n \tilde{\bp}\left(\btheta^{(1)},\btheta^{(2)}\right)  \right)
\label{eq:BitFlipind}
\end{align}

\begin{theorem}
Under the null hypothesis that $\bU$ and $\bV$ are independent with true probability vectors $\bpipi{1}, \bpipi{2} > 0$ respectively, then we have as $n \to \infty$,
$
\min_{\btheta^{(1)},\btheta^{(2)}} \left\{ \bchisqBF{ \btheta^{(1)},\btheta^{(2)};\epsilon} \right\} \stackrel{D}{\to} \chi^2_{(r-1)(c-1)}.
$
\end{theorem}
\begin{proof}
As in the proof of \Cref{thm:ExpInd}, we will use \Cref{thm:KRGeneral}, so we first need to the verify the following equality holds, where $\bp(\btheta) = \btheta^{(1)} \left(\btheta^{(2)}\right)^\transpose$
\begin{align*}
\Pi \covarJL\left(\bp(\btheta)  \right) \Pi \cdot \covarJL\left(\bp(\btheta) \right)^{-1} \cdot \Pi \covarJL\left( \bp(\btheta)  \right) \Pi  = \Pi\covarJL\left( \bp(\btheta)  \right) \Pi.
\end{align*}
Recall that $\Sigma(\bp(\btheta))$ has an eigenvector of all 1's and the other eigenvectors are orthogonal to it.  Hence, when we diagonalize $\Sigma(\bp(\btheta)) = BDB^\intercal$, then $\Pi B$ is going to be the same as $B$ except the column of all 1's becomes zero, which we will denote as $B^*$.  Hence, $\Sigma(\bp(\btheta)) \Pi \Sigma(\bp(\btheta))^{-1} \Pi \Sigma(\bp(\btheta))$, becomes $B D^* B^\intercal$ where $D^*$ is the same as $D$ except the eigenvalue on the diagonal for the eigenvector of all 1's becomes zero. Then projecting the resulting matrix with $\Pi$ will not change the resulting product  We then have 
\begin{align*}
& \Pi \covarJL\left(\bp(\btheta)  \right) \Pi \cdot \covarJL\left(\bp(\btheta) \right)^{-1}\cdot  \Pi \covarJL\left( \bp(\btheta)  \right) \Pi \\
& \qquad = \Pi B D^* B^\intercal\Pi = \Pi \Sigma(\bp(\btheta)) \Pi.
\end{align*}
We then need to show that the following equality holds
\begin{align*}
& \Pi \covarJL\left(\bp(\btheta)  \right) \Pi \cdot \Sigma(\bp(\btheta))^{-1} \Pi \nabla \tilde{\bp}(\btheta^{(1)}, \btheta^{(2)}) \\
& \qquad  = \Pi \nabla \tilde{\bp}(\btheta^{(1)}, \btheta^{(2)}) 
\end{align*}
We again use the fact that $\Sigma(\bp(\btheta))$ has an eigenvector that is all 1's. Hence, we have 
$$
\Pi\Sigma(\bp(\btheta)) \Pi = \Sigma(\bp(\btheta)) - \frac{e^{\epsilon/2}}{rc \cdot ( e^{\epsilon/2} + 1)^2} \pmb{1} \pmb{1}^\intercal.
$$
We then have,
\begin{align*}
&  \Pi \covarJL\left(\bp(\btheta)  \right) \Pi \cdot \Sigma(\bp(\btheta))^{-1} \Pi \nabla \tilde{\bp}(\btheta^{(1)}, \btheta^{(2)}) \\
 & \quad  =  \left( \Sigma(\bp(\btheta)) - \frac{e^{\epsilon/2}}{rc \cdot ( e^{\epsilon/2} + 1)^2} \pmb{1} \pmb{1}^\intercal \right) \\
 & \qquad\qquad \cdot \Sigma(\bp(\btheta))^{-1} \Pi \nabla \tilde{\bp}(\btheta^{(1)}, \btheta^{(2)}) \\
 & \quad = \left( I_{rc} - \frac{1}{rc} \one\one^\intercal  \right) \Pi \nabla \tilde{\bp}(\btheta^{(1)}, \btheta^{(2)}) \\
 & \quad = \Pi \Pi \nabla \tilde{\bp}(\btheta^{(1)}, \btheta^{(2)}) = \Pi \nabla \tilde{\bp}(\btheta^{(1)}, \btheta^{(2)}).
\end{align*}
\end{proof}

\begin{algorithm}
\caption{Local DP IND Test: $\BitIND$}
\label{alg:LocalBitIND}
\begin{algorithmic}
\REQUIRE $(\bbx_1,\cdots, \bbx_n)$, $\epsilon$, $\alpha$.
\STATE Let $\tilde{\bH} = \sum_{i=1}^n \Mbit(\bbx_i,\epsilon)$.  
\STATE $ q = \min_{\bpipi{1},\bpipi{2}} \left\{ \bchisqBF{\epsilon; (\bpipi{1},\bpipi{2})} \right\}$ from \eqref{eq:BitFlipind}.
\STATE {\bf if } $q > \chi_{(r-1)(c-1),1-\alpha}^2$
	 Decision $\gets $ Reject.
\STATE {\bf else }
	 Decision $\gets$ Fail to Reject.
\ENSURE Decision
\end{algorithmic}
\end{algorithm}

We present the test in \Cref{alg:LocalBitIND}.  The following result follows from same privacy analysis as before.
\begin{theorem}
$\BitIND$ is $\epsilon$-LDP.  
\end{theorem}

\subsection{Empirical Results}
As we did for the locally private goodness of fit tests, we empirically compare the power for our various tests for independence.  We consider the null hypothesis that the two sequences of categorical random variables $\{\bU_\ell \}_{\ell = 1}^n$ and $\{\bV_\ell \}_{\ell = 1}^n$ are independent of one another.  Under an alternate hypothesis, we generate the contingency data according to a non-product distribution.  We fix the distribution $\bp^{(1)}$ for the contingency table to be of the following form, where $\bpipi{1} \in \R^r$ is the unknown distribution for $\{\bU_\ell \}_{\ell = 1}^n$, $\bpipi{2} \in \R^c$ is the unknown distribution for $\{\bV_\ell \}_{\ell = 1}^n$, and $r,c$ are even
\begin{align}
& \bp^{(1)} = \bpipi{1} \left( \bpipi{2} \right)^\intercal \nonumber \\
& \quad + \eta (1,-1,\cdots, -1, 1)^\intercal(1,-1,\cdots, -1, 1)
\label{eq:alt_IND}
\end{align}

Note that the hypothesis test does not know the underlying $\bpipi{i}$ for $i \in \{1,2\}$, but to generate the data we must fix these distributions.  We show power results when the marginal distributions satisfy $\bpipi{1} = (1/r, \cdots, 1/r)$ and $\bpipi{2} = (1/c, \cdots, 1/c)$.  In \Cref{fig:power_IND}, we give results for various $n$ and $\epsilon \in \{ 1,2,4\}$ for the case when $(r,c) = (2,2)$ and $\eta = 0.01$ as well as the case when $(r,c) = (10,4)$ and $\eta = 0.005$.  

Note that we only give results for the two tests $\ExpIND$ and $\BitIND$.  

 \begin{figure}[h]
 \begin{subfigure}{.49\textwidth}
   \centering
 \includegraphics[width=.9\textwidth]{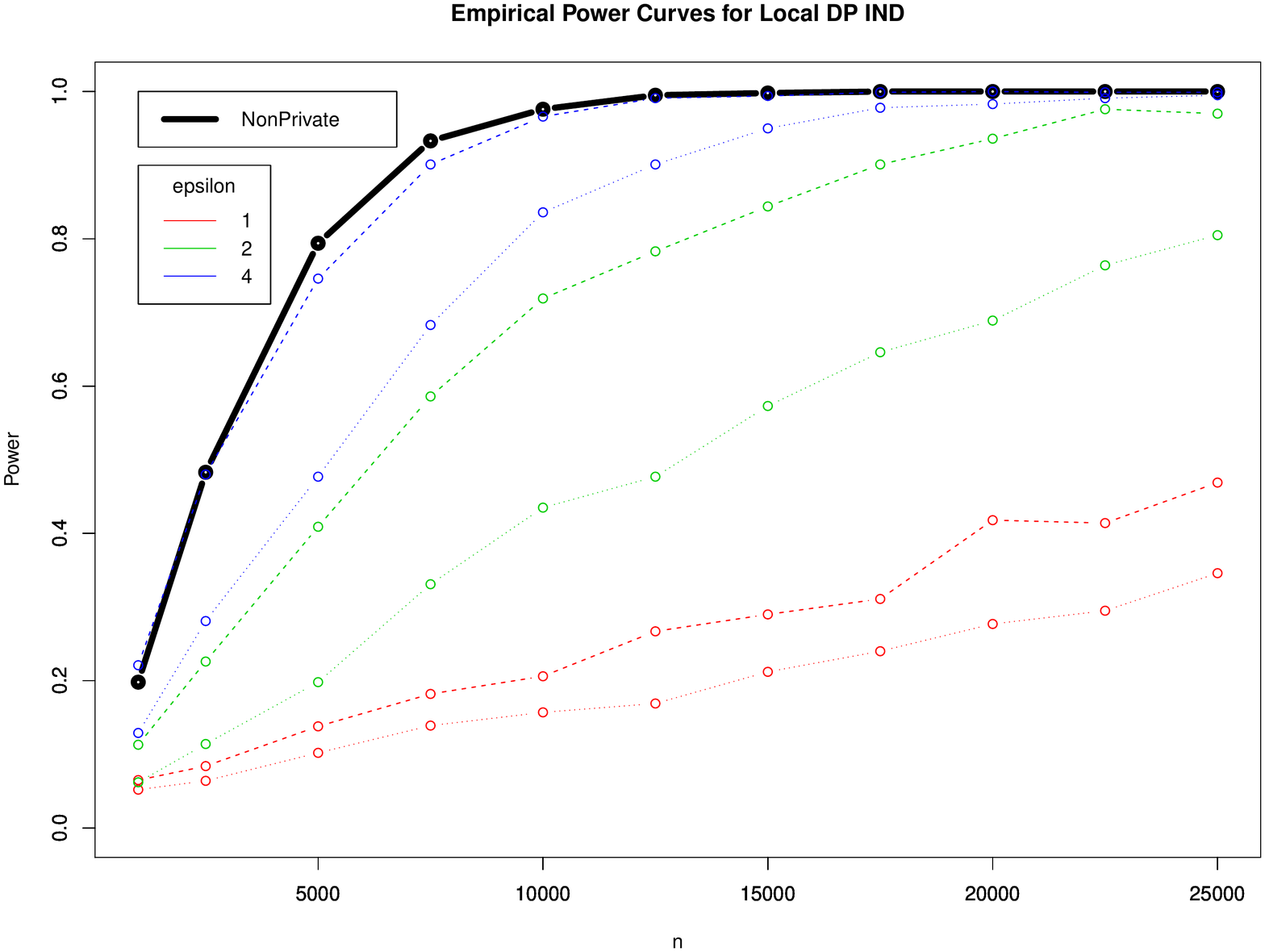}
 \end{subfigure}
 \hfill
 \begin{subfigure}{.49\textwidth}
   \centering
 \includegraphics[width=.9\textwidth]{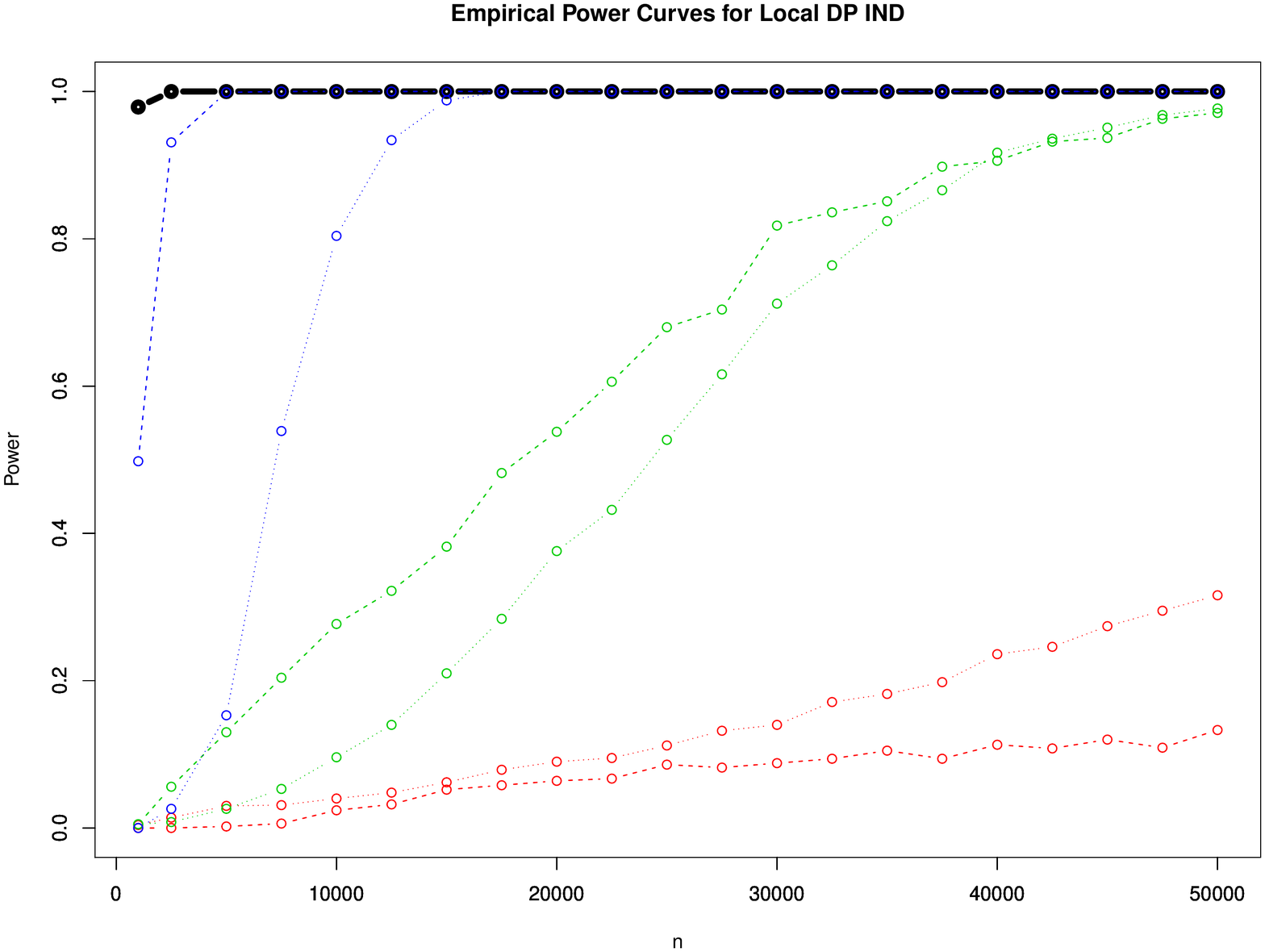}
 \end{subfigure}
 \caption{Comparison of empirical power of classical non-private test versus local private tests $\ExpIND$ (dashed line), and $\BitIND$ (dotted line).  In the left plot we set $(r,c) = (2,2)$ and $\eta = 0.01$ while we set $(r,c) = (10,4)$ and $\eta = 0.005$ in the right plot for the contingency table data distribution given in \eqref{eq:alt_IND}. 
 \label{fig:power_IND}}
 \end{figure}

\section{Conclusion}
We have designed several hypothesis tests, each depending on different local differentially private algorithms: $\GaussGOF$ ($\LapGOF$), $\ExpGOF$, and $\BitGOF$ as well as their corresponding independence tests $\GaussIND$, $\ExpIND$, and $\BitIND$. This required constructing different statistics so that the resulting distribution after injecting noise into the data in order to satisfy privacy could be closely approximated with a chi-square distribution.  Hence, we designed  rules for when a null hypothesis $\nullhyp$ should be rejected while satisfying some bound $\alpha$ on Type I error.  Further, 
We showed that each statistic has a noncentral chi-square distribution when the data is drawn from some alternate hypothesis $\althyp$.  Depending on the form of the alternate probability distribution, the dimension of the data, and the privacy parameter, either $\ExpGOF$ or $\BitGOF$ gave better power.  This corroborates  the results from \citet{KOV14} who showed that in hypothesis testing, different privacy regimes have different optimal local differentially private mechanisms, although utility in their work was in terms of KL divergence.  Our results show that the power of the test is directly related to the noncentral parameter of the test statistic that is used.  This requires the data analyst to carefully consider alternate hypotheses, as well as the data dimension and privacy parameter for a particular test and then see which test statistic results in the largest noncentral parameter.  
We focused primarily on goodness of fit testing, where the null hypothesis is a single probability distribution.  We further developed local private independence tests which resulted in using previous general chi-square theory presented in \cite{KR17}.  This basic framework can be used to develop other chi-square hypothesis tests where the null hypothesis is not a single parameter.  We hope that this will lead to future work on designing local differentially private hypothesis tests beyond chi-square testing.  
 
 \clearpage
 \bibliographystyle{plainnat}
 \bibliography{refs}

\newpage
\appendix
\section{Omitted Proofs}

\begin{proof}[Proof of \Cref{lem:technical}]
This follows from the central limit theorem.  We first compute the expected value
$$
\frac{1}{n}\ex{\tilde{\bH}} = \frac{e^{\epsilon/2}}{e^{\epsilon/2}+1} \bp + \frac{1}{e^{\epsilon/2}+1} (1-\bp).
$$
In order to compute the covariance matrix, we consider the diagonal term $(j,j)$
$$
\ex{(\tilde{H}_j)^2} = \frac{e^{\epsilon/2}}{e^{\epsilon/2}+1} p_j + \frac{1}{e^{\epsilon/2}+1} (1-p_j)
$$
Next we compute the off diagonal term $(j,k)$
\begin{align*}
\ex{\tilde{H}_j \ \tilde{H}_k} & = \left(\frac{e^{\epsilon/2}}{e^{\epsilon/2}+1}\right)^2 \prob{X_j = 1 \ \& \ X_k=1} \\
& \qquad + \frac{e^{\epsilon/2}}{\left(e^{\epsilon/2}+1\right)^2} \left( \prob{X_j = 1 \ \& \ X_k=0} + \prob{X_j = 0 \ \& \ X_k=1}\right)\\
& \qquad + \frac{1}{\left(e^{\epsilon/2}+1\right)^2} \prob{X_j = 0 \ \& \  X_k = 0} \\
& = \left(\frac{1}{e^{\epsilon/2}+1} \right)^2 \left[ e^{\epsilon/2} \left( p_j + p_k\right) + (1-p_j - p_k)\right]\\
\end{align*}
Before we construct the covariance matrix, we simplify a few terms
\begin{align*}
& \ex{(\tilde{H}_j)^2}  - \ex{\tilde{H}_j}^2 \\
&\qquad = \left(\frac{1}{e^{\epsilon/2}+1}  \right)^2 \left[ \left(e^{\epsilon/2} -1\right)\left(e^{\epsilon/2} +1 \right)p_j + 1+e^{\epsilon/2} - \left( \left(e^{\epsilon/2} -1 \right)p_j + 1 \right)^2 \right] \\
& \qquad = \left(\frac{1}{e^{\epsilon/2}+1}  \right)^2 \left[ \left(e^{\epsilon/2} -1\right)\left( e^{\epsilon/2} +1 -2   \right)p_j  - (e^{\epsilon/2} -1)^2\left(p_j\right)^2 + e^{\epsilon/2} \right] \\
& \qquad = \frac{1}{\left( e^{\epsilon/2} + 1\right)^2}\left[  \left(e^{\epsilon/2} -1\right)^2 p_j\left(1-p_j \right) + e^{\epsilon/2}\right]
\end{align*}

Further, we have
\begin{align*}
& \ex{\tilde{H}_j \tilde{H}_k}  - \ex{\tilde{H}_j}\ex{\tilde{H}_k} \\
& \qquad = \left(\frac{1}{e^{\epsilon/2}+1} \right)^2 \left[ \left(e^{\epsilon/2}-1\right) \left( p_j + p_k\right) + 1 - \left( (e^{\epsilon/2}-1) p_j + 1 \right) \left( (e^{\epsilon/2}-1) p_k + 1 \right)\right]  \\
& \qquad = -\left( \frac{e^{\epsilon/2} -1 }{e^{\epsilon/2} + 1} \right)^2 p_j p_k
\end{align*}

Putting this together, the covariance matrix can then be written as 
\begin{align*}
\covarJL(\bp) &= \ex{\tilde{\bH}\left(\tilde{\bH} \right)^\transpose} - \ex{\tilde{\bH}}\left(\ex{\tilde{\bH}} \right)^\transpose \\
& =\left(\frac{e^{\epsilon/2}-1}{e^{\epsilon/2}+1} \right)^2 \left[ \diag\left( \bp \right) - \bp\left( \bp \right)^\transpose \right] +\frac{e^{\epsilon/2}}{\left(e^{\epsilon/2}+1\right)^2} I_d
\end{align*}
\end{proof}

\end{document}